\documentclass[12pt,a4paper,leqno]{amsart}

\usepackage[latin1]{inputenc}
\usepackage[T1]{fontenc}
\usepackage{amsfonts}
\usepackage{amsmath}
\usepackage{amssymb}
\usepackage{eurosym}
\usepackage{mathrsfs}
\usepackage{palatino}
\usepackage{color}
\usepackage{esint}
\usepackage{url}
\usepackage{verbatim}

\usepackage{enumerate}

\newcommand{\R}{\mathbb{R}}
\newcommand{\C}{\mathbb{C}}
\newcommand{\calC}{\mathcal{C}}
\newcommand{\tcalC}{\tilde{\calC}}

\newcommand{\N}{\mathbb{N}}

\newcommand{\Z}{\mathbb{Z}}
\newcommand{\E}{\mathbb{E}}

\newcommand{\diam}{\operatorname{diam}}
\newcommand{\calK}{\mathcal{K}}
\newcommand{\calI}{\mathcal{I}}
\newcommand{\calS}{\mathcal{S}}
\newcommand{\ty}{\tilde{y}}
\newcommand{\calF}{\mathcal{F}}
\newcommand{\bB}{\bar{B}}
\newcommand{\tOmega}{\widetilde{\Omega}}

\newcommand{\wtT}{\widetilde{T}}
\newcommand{\scrD}{\mathscr{D}}
\newcommand{\calM}{\mathcal{M}}

\numberwithin{equation}{section}

\newcommand{\ud}[0]{\,\mathrm{d}}

\newcommand{\dist}[0]{\operatorname{dist}}

\newcommand{\supp}[0]{\operatorname{spt}}
\newcommand{\calD}[0]{\mathcal{D}}
\newcommand{\rest}{{\lfloor}}

\newcommand{\calA}{\mathcal{A}}
\newcommand{\scrX}{\mathscr{X}}

\swapnumbers
\theoremstyle{plain}
\newtheorem{thm}[equation]{Theorem}
\newtheorem{lem}[equation]{Lemma}
\newtheorem{prop}[equation]{Proposition}

\theoremstyle{definition}

\theoremstyle{remark}
\newtheorem{rem}[equation]{Remark}

\title[Non-homogeneous square functions on general sets]{Non-homogeneous square functions on general sets: suppression and big pieces methods} 

\author{Henri Martikainen}
\address[H.M.]{Department of Mathematics and Statistics, University of Helsinki, P.O.B. 68, FI-00014 University of Helsinki, Finland}
\email{henri.martikainen@helsinki.fi}
\thanks{H.M. is supported by the Academy of Finland through the grant
Multiparameter dyadic harmonic analysis and probabilistic methods, and is a member of the Finnish Centre of Excellence in Analysis and Dynamics Research.}

\author{Mihalis Mourgoglou}
\address[M.M.]{Departament de Matem\`atiques, Universitat Aut\`onoma de Barcelona, Edifici C Facultat de Ci\`encies, 08193 Bellaterra (Barcelona), Catalonia}
\email{mourgoglou@mat.uab.cat}
\thanks{Research of M.M. is supported by the ERC grant 320501 of the European Research Council (FP7/2007-2013).}

\author{Emil Vuorinen}
\address[E.V.]{Department of Mathematics and Statistics, University of Helsinki, P.O.B. 68, FI-00014 University of Helsinki, Finland}
\email{emil.vuorinen@helsinki.fi}
\thanks{E.V. is partially supported by T. Hyt\"onen's ERC Starting Grant  Analytic-probabilistic methods for borderline singular integrals, and is a member of the Finnish Centre of Excellence in Analysis and Dynamics Research.}

\makeatletter
\@namedef{subjclassname@2010}{%
  \textup{2010} Mathematics Subject Classification}
\makeatother


\begin{document}

\maketitle

\begin{abstract}
We aim to showcase the wide applicability and power of the big pieces and suppression methods in the theory of local $Tb$ theorems.
The setting is new: we consider conical square functions with cones $\big\{ x \in \R^n \setminus E: |x-y| < 2 \dist(x,E)\big\}$, $y \in E$,
defined on general closed subsets $E \subset \R^n$ supporting a non-homogeneous measure $\mu$. We obtain boundedness criteria in this generality in terms of weak type testing of measures on regular balls
$B \subset E$, which are doubling and of small boundary. Due to the general set $E$ we use metric space methods. Therefore, we also demonstrate the recent techniques from the metric space point of view, and show that they yield the most general known local $Tb$ theorems even with assumptions formulated using balls rather than the abstract dyadic metric cubes. 
\end{abstract}

\section{Introduction}
Square functions are important objects in harmonic analysis and in the theory of PDEs. For example, think of the characterisation of classical Hardy spaces in terms of vertical and conical square function estimates for Littlewood-Paley operators \cite{S}, or the characterisation of the uniformly rectifiable sets in terms of square function estimates of the double gradient of the single layer potential associated with the Laplace operator \cite{DS}. We also consider them important from the framework point of view: they aid in developing novel techniques that work for singular integrals.
In this paper we are interested in a new setup, but also in analysing the very latest \emph{methods of proof} and \emph{characterisations} of the boundedness of square functions (or singular integrals) in the context of non-homogeneous analysis.

For us characterisations of boundedness means various types of sophisticated $Tb$ theorems -- mainly of \emph{local} and also of the so called \emph{big pieces} type.
The latter big pieces type is not equally well-known, but has featured
prominently in many recent important articles, for example in connection with the breakthroughs related to the rectifiability of the harmonic measure \cite{AHMMMTV}.
Such a big pieces $Tb$ was first proved by Nazarov, Treil and Volberg \cite{NTV}  (for Cauchy integral type operators) in connection with Vitushkin's conjecture.
Their point is roughly the following: the assumptions are much weaker than in the usual $Tb$ theorems, but so are the conclusions in that one
only gets the boundedness of the given operator on some big piece.
These theorems have lately become more and more important, and there is also a highly useful connection between them and the local $Tb$ theorems, as observed by us in \cite{MMV}.
For a relatively detailed account of the entangled history of the local $Tb$ theorems we refer to \cite{MMV}.

Lowering the integrability of the appearing test functions is the key problem in local $Tb$ theorems, see e.g. \cite{AHMTT}, \cite{AR}, \cite{AY}, \cite{Ho1}, \cite{HN}, \cite{MMT}, \cite{MMV}.
Related to this and other issues, we developed in \cite{MMV} a new method to prove local $Tb$ theorems via the big pieces
and good lambda methods. Let us explain the main steps of our method later. What it allows, however, is weak type $(1,1)$ testing conditions and improving known results -- both doubling and non-doubling -- in other ways. The method is important regarding
Calder\'on--Zygmund operators too, see \cite{MMT} for the best known integrability exponents related to Hofmann's problem.
A major technical convenience of the method is that the difficult core $Tb$ argument can be carried out in the big pieces $Tb$ part, and is hence of $L^2$ nature even if the original test
functions are not.

In this paper we give the full technical execution of our method in the novel setting of
conical square functions defined on general closed subsets $E \subset \R^n$ supporting a non-homogeneous measure $\mu$. This is the first time non-homogeneous analysis is being
carried out in this particular context. It is an interesting setting with new difficulties of its own,
but also provides us an opportunity to show a full breakdown of the required components in a very general, essentially metric in nature, setting.
The setup includes two measures: one on $E$ and one on the complement $\R^n \setminus E$. Related Ahlfors--David regular theory has been developed in e.g. \cite{HMMM}
by Hofmann, Mitrea, Mitrea and Morris and in \cite{GM} by one of us and A. Herran. Our methods are completely different, not only because of the non-homogeneous setting,
but also because we prove more advanced local $Tb$ theorems using the latest methods of \cite{MMV}.
In this paper the measure $\mu$ living on $E$ can be non-doubling, but the measure used to integrate
over the complement $\R^n \setminus E$ is just an appropriately weighted Lebesgue measure. For the exact setting see Section \ref{sec:setting}.

Let us still highlight an additional benefit of our method that surfaces from this paper. It is the following technical aspect related especially to the metric space theory of local $Tb$ theorems.
With the most general possible local $Tb$ theorems (particularly with Calder\'on--Zygmund operators, general integrability exponents or non-homogeneous measures)
the passage between having the test functions $b_Q$ on some
(completely abstract) metric dyadic cubes $Q$ or having the test functions $b_B$ on metric balls $B$ does not appear to be straightforward, see e.g. the paper by Auscher--Routin \cite{AR}.
The reasons are technical: the operator testing condition does not always seem to be trivial to transfer (from balls to cubes) with general $L^p$ integrability exponents.
This is because the Hardy inequality might not be available due to some irregular underlying metric measure space, or even if it is, it is not useful
with certain exponent (which is called the super-dual case, see e.g. \cite{AR}).
Even the accretivity condition is a problem with non-homogeneous measures. But the old proofs seem to require the existence of $b_Q$ on dyadic cubes (this is required to build some dyadic
martingales).
For these reasons some of these theorems have been previously formulated involving these rather abstract cubes even in the statements.
A point we like to make is that these new methods are flexible enough in that
we can easily state and prove our extremely general local $Tb$ having the existence of the test functions $b_B$ only on very regular balls $B$.

\subsection{The setting}\label{sec:setting}
Let $E$ be a closed subset of $\R^n$ and $S \colon \R^n \setminus E \times E \to \C$ be a kernel, which, for some fixed $\alpha, \beta \in (0,1]$, $m>0$, and $K_1, K_2>0$, satisfies the size estimate
\begin{equation}\label{eq:size}
|S(x,y)| \leq K_1 \frac{1}{|x-y|^{m+\alpha}},
\end{equation}
for all $(x,y) \in \R^n \setminus E \times E$, and the $y$-H\"older estimate
\begin{equation}\label{eq:yHol}
| S(x,y)-S(x,y') | \leq K_2 \frac{|y-y'|^{\beta}}{|x-y|^{m+\alpha+\beta}}, 
\end{equation}
where $x \in \R^n \setminus E, y,y' \in E$ and  $|y-y'| \leq \frac{|x-y|}{2}$.  If $\mu$ is a non-negative finite measure with support in $E$, or if $\mu$ is a measure of order $m$ in $E$  (see definition below), we define the integral operator $T_\mu$ on $\bigcup_{p \in [1,\infty]}L^p(\mu)$ by
$$
T_\mu f(x) := \int_E S(x,z)f(z) \ud \mu(z), \quad x \in \R^n \setminus E. 
$$
Note that the integral is absolutely convergent. Let  $\sigma$ be a measure  in $\R^n \setminus E$ given by
$$
\sigma(A) := \int_A \frac{\ud m_n(x)}{d(x,E)^n},
$$
where $A \subset \R^n \setminus E$ is any Lebesgue measurable set. Here $d(x,E)$ stands for the distance of the point $x$ to $E$ and $m_n$ is the Lebesgue measure in $\R^n$. For a point $y \in E$, we denote the \emph{cone} $\Gamma(y)$ at $y$ by
$$
\Gamma(y):=\big\{ x \in \R^n \setminus E: |x-y| < 2 d(x,E)\big\}.
$$

We can now define the \emph{conical square function} $\calC_\mu$ on $\bigcup_{p \in [1,\infty]} L^p(\mu)$ by setting
$$
\calC_\mu f(y):= \Big( \int_{\Gamma(y)} |T_\mu f(x)|^2 d(x,E)^{2\alpha} \ud \sigma(x) \Big)^\frac{1}{2}, \quad y \in E.
$$
Note that we will also use square functions with other measures than $\mu$. If $\mathcal{M}(E)$ is the set of complex measures with support in $E$, then for $\nu \in \calM(E)$, $x\in \R^n \setminus E$ and $y \in E$ we define
$$
T\nu(x)  :=\int_E S(x,z) \ud \nu(z)
$$
and 
$$
\calC \nu (y) :=\Big( \int_{\Gamma(y)} |T \nu (x)|^2 d(x,E)^{2 \alpha} \ud \sigma(x) \Big)^\frac{1}{2}.
$$

We will need to use various truncated versions of the square function. For any $t>0$ and $y \in E$ define the truncated cones 
$$
\Gamma_t(y):= \{x \in \Gamma(y) \colon d(x,E) > t\}\quad
\textup{and}\quad
\Gamma^t(y) :=\{ x \in \Gamma(y) \colon d(x,E) \leq t\}.
$$
If $0<s<t$ we set $\Gamma_s^t(y):= \Gamma_s(y) \cap \Gamma^t(y)$. The corresponding truncated square functions, defined with integration over the truncated cones only, are denoted for instance by $\calC^t$ or $\calC_{\mu,s}^t$, depending on the situation. As an example we have
$$
\calC_\mu^t f(y):= \Big( \int_{\Gamma^t(y)} |T_\mu f(x)|^2 d(x,E)^{2\alpha} \ud \sigma(x) \Big)^\frac{1}{2}, \quad y \in E.
$$

Before going any further, let us introduce some notation which is necessary for the statement of our main theorem.

\subsection{Notation and key definitions}

An open ball in $\R^n$ with center $x \in \R^n$ and radius $r>0$ is denoted by $B(x,r)$, while an open ball in $E$ with center $y \in E$ and radius $r>0$ is defined as $B_E(y,r):=B(y,r) \cap E$. Often, when there is no danger of confusion, we  may drop the subscript $E$. We  write $\bB(x,r)$ and $\bB_E(y,r)$  for the corresponding closed balls and  $r(B)$ for the radius of the ball $B$. If we talk about a ball without specifying whether it is open or closed, then it should be understood that it can be either one.

For two constants $A,B\geq 0$ the notation $A \lesssim B$ means that there exists an absolute constant $C>0$ such that $A \leq CB$. We write $A \lesssim_{\alpha, \beta, \dots} B$ to indicate the dependence of the constant $C$ on $\alpha, \beta$, etc. Two sided estimates $A \lesssim B \lesssim A$ are abbreviated as $A \sim B$. 

If $ \nu \in \calM(E)$, then $|\nu|$ denotes its {\it total variation}. We say that a measure $\mu$ is {\it of order $m$ in $E$} if $\mu$ is non-negative, $\supp \mu \subset E$  and 
\begin{equation}\label{order m}
\mu(B (y,r)) \lesssim r^m
\end{equation}
holds for all $y  \in E$ and $r>0$. If $\mu$ is a non-negative Borel measure or a complex Borel measure in $\R^n$ and $F \subset \R^n$ is a Borel set, then $\mu \lfloor F$ is a measure defined by $\mu \lfloor F (G)= \mu(F \cap G)$, where $G$ is any Borel set.

Let $a,b \geq1,\kappa>0$ and suppose $\mu$ is a non-negative Borel measure in $\R^n$.  A ball $B(x_0,r)$ is said to be {\it $(a,b)$-doubling} (with respect to $\mu$) if 
\begin{equation}\label{def. of doubling}
\mu(B(x_0,a r)) \leq b \mu(B(x_0,r)).
\end{equation}
The ball $B(x_0,r)$ is said to have $\kappa$-small boundary if for all $s \in [0,1]$ there holds that
\begin{equation}\label{def. of small boundary}
\mu\big(\{x \in \R^n \colon r(1-s) < |x-x_0|<r(1+s)\}\big) \leq \kappa s \, \mu(B(x_0,3r)).
\end{equation}
We say that a ball $B_E(y,r)$ is $(a,b)$-doubling or that it has {\it $\kappa$-small boundary} if $B(y,r)$ has the corresponding property.  These concepts are defined similarly with closed balls just by replacing the open balls in \eqref{def. of doubling} and \eqref{def. of small boundary} with closed balls (the left hand side of \eqref{def. of small boundary} stays the same).
 
\subsection{Statement of the main theorem and further discussion}
We will next formulate our main local $Tb$ theorem for $\calC_\mu$. First, for experts and non-experts alike let us try to identify the main points:
\begin{itemize}
\item A testing measure $\nu_B$ is given on each \emph{regular} (i.e. doubling and of small boundary) surface \emph{ball} $B \subset E$.
\item The testing measure (or somewhat less generally the testing function) $\nu_B$ is required to be supported on $B$ and accretive (normalised to the condition $\nu_B(B) = \mu(B)$).
What is important is that it need not satisfy strong estimates, only the $L^1$ type condition $|\nu_B|(B) \leq C_1 \mu(B)$ and some rather weak \emph{quantified} absolute continuity assumption (assumption 4) below). For example, the latter technical condition
is automatically satisfied by H\"older's inequality should the measure be a function $\nu_B = b_B \,d\mu$ satisfying a slightly stronger $L^{1+\epsilon}$ type condition. That is why it does
not appear in more classical formulations.
\item The testing condition on the operator side is completely decoupled from the regularity assumption on $\nu_B$. We only require the weak type estimate
$\sup_{\lambda>0} \lambda^s \mu\big(\{ y \in B: \calC^{r(B)} \nu_B (y) > \lambda \} \big) \leq C_2 |\nu_B |(B)$ for some $s > 0$ (e.g. $s = 1$). In fact, we may only require
this estimate outside some small enough exceptional set $U_B \subset B$.
\item Using weak type testing is also in line with keeping the condition necessary for the boundedness:
$L^2$ boundedness implies such a condition with $s=1$ using the boundedness of $\calC$ from the set of finite measures
to $L^{1,\infty}(\mu)$.
\end{itemize}

\begin{thm}\label{thm:main}
Suppose $\mu$ is a measure of order $m$ in $E$. Let $b, \kappa>0$ be big enough constants depending only on the dimension $n$, and let $\varepsilon_0 \in (0,1)$ and  $C_1>0$ be given constants. Assume that for every $(10, b)$-doubling closed ball $B$ in $E$ with a $\kappa$-small boundary (with respect to $\mu$) there exists a complex measure $\nu_B \in \calM(E)$ with the following properties:
\begin{enumerate}
\item $\supp \nu_{B} \subset B$;
\item $\nu_B(B) = \mu(B)$; 
\item $|\nu_B|(B) \leq C_1 \mu(B)$;
\item If $A \subset B$ is a Borel set with $\mu(A) \leq \varepsilon_0 \mu(B)$, then $|\nu_B|(A) \leq  \frac{1}{16C_1} |\nu_B|(B)$.
\end{enumerate}

Furthermore, assume that we are given constants $s,C_2>0$, and in every ball $B$ as above a set $U_B \subset B$, so that
\begin{enumerate}[a)]
\item $| \nu_B | (U_B) \leq  \frac{1}{16C_1} |\nu_B | (B)$;
\item $\sup_{\lambda>0} \lambda^s \mu\big(\{ y \in B \setminus U_B: \calC^{r(B)} \nu_B (y) > \lambda \} \big) \leq C_2 |\nu_B |(B)$.
\end{enumerate}

Under these assumptions the square function $\calC_\mu$ is bounded in $L^p(\mu)$ for every $p \in (1,\infty)$ with norm depending on $p$ and the preceding constants.
\end{thm}

\begin{rem}\label{rem:findiam}
When $\diam(E) < \infty$ one can restrict the testing surface balls to have radius $r \in (0, \diam(E)]$. But then one needs to interpret
$E$ to be one of the balls so that $\nu_E$ exists. 
\end{rem}

Such a theorem was obtained in \cite{MMV} for vertical square functions in the upper half space $\R^{n+1}_+$ and in \cite{MMT} for maximal truncations of Calder\'on-Zygmund operators. Although the method is the same, to prove Theorem \ref{thm:main} in this generality, one has to overcome non-trivial issues stemming from the geometrically complicated environment on which our objects are defined.

It is not completely evident which measure $\sigma$ one should or could use on the complement $\R^n \setminus E$ to define the integration over the cones $\Gamma(y) \subset \R^n \setminus E$, $y \in E$.
In \cite{MMO} we used $\sigma = \mu \times dt/t^{m+1}$ with conical square functions
in the upper half-space $\R^{k+1}_+$ and with a measure $\mu$ of order $m$ supported in $\R^k$ (this is essentially the setting $n = k+1$ and $E = \R^k \times \{0\}$).
A key requirement is that $\sigma$ ought to be a measure on $\R^n \setminus E$
that assigns a bounded measure to Whitney regions associated with $E$ (like the measure $\sigma = \mu \times dt/t^{m+1}$ in the upper-half space does:
$\sigma(B \times (r(B)/2, r(B)) \lesssim \mu(B) / r(B)^m \lesssim 1$).
For some parts of the theory there is another important aspect in play: we also need to be able to exploit the fact that
nearby cones have some geometric cancellation and then see this appropriately on the $\sigma$ measure side (this also works with $\sigma = \mu \times dt/t^{m+1}$ in the upper half-space
because of the appropriate Lebesgue part).
For this latter reason we use just the Lebesgue measure weighted with $d(x,E)^{-n}$ here i.e.
for us $\sigma = d(x,E)^{-n}\,dm_n$ like defined above. As is apparent from the upper half-space situation, in some scenarios multiple other natural choices also work.
Further investigation on this issue could be appropriate -- but for us the main thing was to remove regularity assumptions on $\mu$.
Compare also to \cite{HMMM} where two different ADR measures are used.

 \subsection{Outline of the method and proof}
 The key steps of our method are as follows:
 \begin{enumerate}
 \item The non-homogeneous good lambda method by Tolsa \cite{ToBook}. This can simply be thought of as a highly flexible way to glue some local results
 in to the desired global result. It says that it is enough to find in every $(10,b)$-doubling closed ball $B$ in $E$ with a $\kappa$-small boundary a subset $G_B \subset B$ with $\mu(G_B) \gtrsim \mu(B)$, where
$$
 \calC \colon \calM(E) \to L^{1,\infty}(\mu\lfloor G_B)
 $$
 is bounded with a constant $C$ that is independent of $B$, that is, the inequality
 $$
 \mu\big( \{ y \in G_B \colon \calC \nu (y) > \lambda\} \big) \leq \frac{C}{\lambda} |\nu|(E)
 $$
 holds for all $\lambda>0$ and $\nu \in \calM(E)$. We give a proof in our context in Section \ref{sec:good lambda}.
 \item  Prove a big pieces global $Tb$: Theorem \ref{the big piece Tb}. This is interesting on its own right but is also the key step in proving the required local result needed to apply
 the non-homogeneous good lambda method. The theorem says that given a closed ball $B$, a finite measure $\mu$ with support in $B$ and a function $b \in L^\infty(\mu)$, then under certain very weak  assumptions one finds a subset $G \subset B$ with $\mu(G) \gtrsim \mu(B)$ so that the square function $\calC_\mu$ is bounded in $L^2(\mu \lfloor G_B)$. The proof relies on the idea of \emph{suppression} -- an amazing technique of Nazarov--Treil--Volberg designed to cook up an operator that behaves significantly better than the original one but also agrees with the original one on a large enough set.
In this paper we show how to suppress these conical square functions.
\item The third step is to prove that the assumptions of the non-homogeneous good lambda method hold in our situation.
This is Proposition \ref{prop:main}.
So we take an arbitrary closed $(10,b)$-doubling ball $B$ in $E$ with a $\kappa$-small boundary. Then, by assumption, we have the test measure $\nu_B$. By writing the polar decomposition $ \nu_B= b   |\nu_B|$, where $|b(y)|=1$ for all $y \in B$, we have
$$
\calC\nu_B= \calC_{|\nu_B|}b.
$$  
Next, we want to use the big pieces global $Tb$ theorem with the measure $|\nu_B|$ and the bounded function $b$. Using stopping times we can do this, and we get a set $G_B$ where $\calC_{|\nu_B|}$ is bounded in $L^2(|\nu_B| \lfloor G_B)$.  This will be done in such a way that $|\nu_B|\lfloor G_B = \varphi \mu \lfloor G_B$ for some function $\varphi \sim 1$, whence we can come back to the measure $\mu$ and derive the desired result that $\calC_\mu$ is bounded in $L^2(\mu \lfloor G_B)$. A weak $(1,1)$ argument is required to conclude that the assumptions of the good lambda theorem hold.
 \end{enumerate} 
We also need various preliminaries (Section \ref{preliminaries}). Some geometric considerations related to cones are presented in Section \ref{geometric problem}.
Certain technical details are also given in the Appendixes. 
We use metric dyadic cubes and their randomisation purely as a technical tool (the cubes do not appear in the statement of the main theorem). This is needed
for the $Tb$ argument in the big pieces $Tb$ theorem.
The randomisation
originates from the paper by Hyt\"onen--Martikainen \cite{HM} but we use the latest version from Auscher--Hyt\"onen \cite{AH} with some further modifications.
 
\section{Preliminaries}\label{preliminaries}  

We begin by collecting here some notation and  basic estimates that will be used in the later sections.  We shall use maximal functions on the set $E$.  Let $\mu$ be a non-negative locally finite Borel measure in $E$. The {\it centred} Hardy-Littlewood maximal operator is defined for any non-negative locally finite Borel measure $\nu$ with support in $E$ by   
$$
M_\mu \nu(y):= \sup_{r>0} \frac{\nu (B(y,r))}{\mu(B(y,r))}, \quad y \in E,
$$
and the {\it radial} one by 
$$
M^m \nu(y):= \sup_{r>0} \frac{\nu (B(y,r))}{r^m}, \quad y \in E.
$$
If  $\nu  \in \mathcal M(E)$, we set $M_\mu \nu:= M_\mu|\nu|$ and $M^m \nu:= M^m |\nu|$. If $f\in L^1_{loc}(\mu)$ we write $M_\mu f:=M_\mu (|f|\mu)$.

 A basic property of $M_\mu$ is that it is bounded from $\calM(E)$  into $L^{1,\infty}(\mu)$ and from $L^p(\mu)$ into $L^p(\mu)$, $p \in (1,\infty)$. If $\mu$ is of order $m$ there holds that $M^m\nu(y) \lesssim M_\mu \nu(y)$ for every $y$ and $\nu \in \mathcal M(E)$, and therefore the boundedness properties of $M_\mu$ transfer to $M^m$.

\begin{lem}\label{annulus calculation}
Let $y \in E$, $r>0$ and suppose $\mu$ is a non-negative Borel measure with $\supp \mu \subset E$. If $f \in L^1_{loc}(\mu)$, then 
$$
 \int_{E} \frac{|f(z)| \ud \mu (z)}{(r+ |z-y|)^{m+\alpha}} \lesssim r^{-\alpha}M^m(|f|\mu)(y),  
$$
where the implicit constant is independent of $r$ and $y$.
\end{lem}

The proof of Lemma \ref{annulus calculation} is a standard calculation dividing the integration area as 
$$
E = \big( B_E(y,r)\big) \cup \bigcup_{k=0}^\infty \big( B_E(y,2^{k+1}r)\setminus B_E(y,2^kr)\big).
$$  
 
The next lemma is a simple geometric observation:

\begin{lem}\label{distance to a point in E}
Let $y,z \in E$ and $x \in \Gamma(y)$. Then
$$
|x-z| \sim |x-y|+|y-z|.
$$
\end{lem}
\begin{proof}
If $|y-z| \geq 2|x-y|$, then 
$$
|x-z| \geq |y-z|-|x-y| \sim |y-z|+|x-y|.
$$ 
On the other hand if $|y-z|<2|x-y|$, then 
$$
|x-z| \geq d(x,E) >\frac{1}{2}|x-y| \sim |x-y|+|y-z|.
$$
\end{proof}

\begin{lem}\label{whitney section of a cone}
For every $y \in E$ and $t>s>0$ there holds that
$$
\sigma\big(\Gamma^t_{s}(y)\big) \lesssim \left(\frac{t}{s}\right)^n.
$$
\end{lem}

\begin{proof}
Fix some $y \in E$ and  $t>s>0$. If $x \in \Gamma_{s}^{t} (y)$, then we have 
$$
|x-y| < 2d(x,E) \leq 2t.
$$
Also, $x$ satisfies $d(x,E) \geq s$. From these we directly get that
$$
\sigma\big(\Gamma^t_{s}(y)\big) = \int_{\Gamma_{s}^t(y)} \frac{\ud m_n(x)}{d(x,E)^n} \leq \frac{m_n(B(y,2t))}{s^n} \sim \left( \frac{t}{s}\right)^n.
$$
\end{proof}

From Lemma \ref{whitney section of a cone} we get the following two lemmas:

\begin{lem}\label{cone computation}
Let  $s,t>0$. Then for every $y \in E$ and $r>0$ there holds that
$$
\int_{\Gamma(y)} \frac{d(x,E)^s}{(d(x,E)+r)^{s+t}} \ud \sigma(x) \lesssim_{s,t} r^{-t}.
$$
\end{lem}

\begin{proof}
Fix some $y \in E$ and $r>0$. Then
\begin{equation*}
\int_{\Gamma^r(y)}  \frac{d(x,E)^s}{(d(x,E)+r)^{s+t}} \ud \sigma(x) \leq \sum_{k=0}^{\infty} \sigma\big(\Gamma^{2^{-k}r}_{2^{-k-1}r}(y) \big) \frac{(2^{-k}r)^s}{r^{s+t}}
\lesssim_s \frac{1}{r^t}
\end{equation*}
and
\begin{equation*}
\int_{\Gamma_r(y)}  \frac{d(x,E)^s}{(d(x,E)+r)^{s+t}} \ud \sigma(x) \leq \sum_{k=0}^\infty \sigma\big( \Gamma_{2^kr}^{2^{k+1}r} (y) \big) \frac{1}{(2^kr)^t} \lesssim_t \frac{1}{r^t}.
\end{equation*}
\end{proof}

\begin{lem}\label{truncated is bounded}
Suppose $\mu$ is a measure of order $m$ in $E$. For every  $t>s>0$ and $p \in (1,\infty)$ the truncated square function $\calC_{\mu,s}^t$ is bounded in $L^p(\mu)$.
\end{lem}

\begin{proof}
Fix some numbers $0<s<t$ and $p \in (1,\infty)$, and suppose $f \in L^p(\mu)$.  If $y \in E$ and $x \in \Gamma(y)$, then Lemmas \ref{annulus calculation} and \ref{distance to a point in E} give 
\begin{equation}\label{T dominated by M}
\begin{split}
|T_\mu f(x)| \lesssim \int_E \frac{|f(z)|}{|x-z|^{m+\alpha}}   \ud \mu(z) & \sim \int_E \frac{|f(z)|}{(|x-y|+|y-z|)^{m+\alpha}} \ud \mu(z) \\
& \lesssim |x-y|^{-\alpha} M^m(|f| \mu)(y) \\ 
&\lesssim |x-y|^{-\alpha}M_\mu f (y).
\end{split}
\end{equation}
Combining this with Lemma \ref{whitney section of a cone} yields
\begin{equation*}
\begin{split}
\calC_{\mu,s}^t f(y) = \Big( \int_{\Gamma_s^t(y)} | T_\mu f(x)|^{2} d(x,E)^{2\alpha} \ud \sigma(x) \Big)^\frac{1}{2} &\lesssim \sigma(\Gamma_s^t(y))^\frac{1}{2}M_\mu f(y) \\ &\lesssim_{s,t}  M_\mu f(y).
\end{split}
\end{equation*}

Thus $\calC_{\mu,s}^t f$ is pointwise dominated by $M_\mu$, and the claim follows from boundedness of $M_\mu$ in $L^p(\mu)$.
\end{proof}

\subsection*{Random dyadic cubes}
In the big piece global $Tb$ theorem and in the Whitney decomposition related to the good lambda method  we shall use systems of dyadic cubes in $E$. Even though we are in $\R^n$, the existence of these is a metric space argument, because there is no direct way of building dyadic systems in \emph{arbitrary} closed subsets of $\R^n$. The specific construction we use is from \cite{AH}, and we elaborate here on how we use it. 

Let $x \in  E$ be a fixed point. The construction of dyadic cubes begins by choosing a set of \emph{reference dyadic points}. Write $ \delta = \frac{1}{1000}$. Set $\scrX_0$ to be any maximal $1$-separated subset of $E$ so that $x \in \scrX_0$. Let $K \in \Z, K \geq 0,$ and assume that the sets $\scrX_k$ and $\scrX_{-k}$ have been chosen for $k\in \{0,\dots,K\}$. Define $\scrX_{K+1}$ to be any maximal $\delta^{K+1}$-separated subset of $E$ such that $\scrX_K \subset \scrX_{K+1}$, and $\scrX_{-K-1}$ to be any maximal  $\delta^{-K-1}$-separated subset of $\scrX_{-K}$ so that $x \in \scrX_{-K-1}$. Continue this way to get the collections $\scrX_k$ for $k \in \Z$. Then $\scrX:= \bigcup_k\scrX_k$ is the set of reference dyadic points. A point in $\scrX$ is denoted by $x^k_\alpha$, where $k$ indicates that $x^k_\alpha \in \scrX_k$, and $\alpha$ indexes the different points in $\scrX_k$. This is precisely as in \cite{AH}, except that here we require that $x \in \scrX_k$ for every $k$. The role of this fixed point will be explained below.

The rest of the construction we take directly as in \cite{AH}. We have a probablity space $\Omega=(\{0, \dots,L\} \times \{1, \dots,M\})^\Z$, where the numbers $L$ and $M$ are related to the properties of $E$ as a geometrically doubling space. That $E$ is geometrically doubling means that there exists a constant $N$ such that every ball $B$ in $E$ with radius $r$ contains at most $N$ points whose distances from each other are at least $r/2$. The set $\Omega$ is equipped with the natural $\sigma$-algebra and the probability measure $\mathbb{P}$ so that the coordinate mappings 
$$
\omega \mapsto \omega (k) \in \{0, \dots, L\} \times \{1, \dots ,M\},
$$ 
where $k \in \Z$, are independent and uniformly distributed over the finite set $\{0, \dots, L\} \times \{1, \dots ,M\}$.

With every $\omega \in \Omega$  there is associated  a set $\{z^k_\alpha (\omega)\}_{k,\alpha}$ of slightly shifted reference dyadic points. To every $z^k_\alpha(\omega)$ corresponds a dyadic cube $Q^k_\alpha (\omega) \subset E$, and $\calD_k(\omega)$ is the collection all cubes $Q^k_\alpha (\omega)$ with the fixed generation $k$. The dyadic lattice $\calD(\omega)$ is $\bigcup_{k \in \Z}\calD_k(\omega)$. To be precise, in \cite{AH} certain open and closed dyadic cubes are constructed, and from these we form our cubes using finite unions and intersections as is done in \cite{HM}, Theorem 4.4.

We list a few relevant properties of the dyadic systems and introduce some notation that will be used later. Let $\omega \in \Omega$.
\begin{itemize}

\item Every dyadic cube $Q^k_\alpha(\omega)$ is a measurable subset of $E$ such that 
\begin{equation}\label{new centers}
B_E(z^k_\alpha(\omega), \frac{\delta^k}{6}) \subset Q^k_\alpha(\omega) \subset B_E(z^k_\alpha(\omega),6 \delta^k).
\end{equation}
We call the point $z^k_{\alpha}(\omega)$ the center of $Q^k_\alpha(\omega)$, and we set $\ell(Q^k_\alpha(\omega)):= \delta^k$ to be the ``sidelength'' of $Q^k_\alpha(\omega)$.
Also, the original reference points satisfy
\begin{equation}\label{center of a cube}
B_E(x^k_\alpha, \frac{\delta^k}{8}) \subset Q^k_\alpha(\omega) \subset B_E(x^k_\alpha,8\delta^k).
\end{equation}

\item The cubes in a given generation are pairwise disjoint and cover the whole set $E$, that is, $Q^k_\alpha(\omega) \cap Q^k_\beta(\omega) = \emptyset$ if $\alpha \not= \beta$ and $E= \bigcup_{\alpha} Q^k_\alpha(\omega)$ for every $k \in \Z$.

\item The dyadic cubes are nested in the sense that for any two cubes $Q^k_\alpha(\omega)$ and $Q^l_\beta(\omega)$ one of the following holds: $Q^k_\alpha(\omega) \cap Q^l_\beta(\omega) = \emptyset, Q^k_\alpha(\omega) \subset Q^l_\beta(\omega)$ or $Q^l_\beta(\omega) \subset Q^k_\alpha(\omega)$.

\item If $Q^k_\alpha(\omega) \in \calD(\omega)$ we denote by ch$(Q^k_\alpha(\omega))$ the collection of cubes $Q^{k+1}_\beta \in \calD_{k+1}(\omega)$ such that $Q^{k+1}_\beta(\omega) \subset Q^k_\alpha(\omega)$. These are called \emph{children} of $Q^k_\alpha(\omega)$.
\item $\widehat{Q^k_\alpha}(\omega)$ denotes the unique cube in $\calD_{k-1}(\omega)$ that contains $Q^k_\alpha(\omega)$. 
\item In the rest of the paper we usually write $Q$ in place of $Q^k_\alpha(\omega)$. Nevertheless, one should keep in mind that this is only a short hand, and there is always the specified generation $k$ such that $Q \in \calD_k(\omega)$. This is important because it may happen that $Q^k_\alpha(\omega)=Q^l_\beta(\omega)$ as \emph{sets in $E$} even if $k \not=l$. In the summations over dyadic cubes below, we are always summing over pairs $(k,\alpha)$. If $Q=Q^k_\alpha(\omega)$, then its center $z^k_\alpha(\omega)$ will be denoted by $c_Q$.

\item Let $B$ be a ball in $E$ with center $y \in E$ and radius $r$. Construct the random dyadic systems $\calD(\omega)$ in $E$ with the initial requirement that $y\in \bigcap_{k \in \Z}\scrX_k$. For any $k \in \Z$ there exists $\alpha(k)$ so that $y=x^k_{\alpha(k)}$.   
Specify $k_0 \in \Z$ by the condition $r< \frac{\delta^{k_0}}{8} \leq \delta^{-1}r$.  
Then define
$$
Q_B(\omega) := Q^{k_0}_{\alpha(k_0)}(\omega)
$$
and 
\begin{equation}\label{random system}
\calD_{B}(\omega):= \big\{ Q^l_\beta(\omega) \in \calD(\omega): l\geq k_0, Q^l_\beta(\omega) \subset Q_{B}(\omega)\big\}.
\end{equation}
With these definitions we have $B \subset Q_B(\omega)$  by \eqref{center of a cube}.

\end{itemize}

We shall also use a variant of the notion of \emph{good dyadic cubes} introduced first by Nazarov, Treil, and Volberg  \cite{NTV1}, and then used in the metric space setting for instance in \cite{HM}. Let again $B$ be some ball in $E$. Using the center of $B$ as the fixed reference dyadic point construct the 
dyadic systems $\calD_B(\omega) \subset \calD(\omega), \omega \in \Omega$, as described above.  Fix some $\omega_0 \in \Omega$. Let $\gamma= \frac{\alpha}{2(m+\alpha)}$ and $r \in \Z, r>0$.  A cube $R \in \calD(\omega_0)=:\calD_{0}$ is said to be \emph{$\calD_B(\omega)$-good} (with parameters $(\gamma,r)$) if for  all $Q \in \calD_B(\omega)$ with $\ell(Q) \geq \delta^{-r}\ell(R)$ there holds that 
\begin{equation}\label{def. of goodness}
\max \big(d(R,Q), d(R, E \setminus Q) \big) \geq \ell(R)^\gamma \ell(Q)^{1-\gamma}.
\end{equation} 
Otherwise $R$ is said to be $\calD_B(\omega)$-\emph{bad}. Note that the systems $\calD_0$ and $\calD(w)$ depend on the ball $B$, but later when we use these systems it should be clear what the ball is. Also, with our definition, every cube $R \in \calD_0$ with $\ell(R) \geq \delta^r\ell(Q_B(\omega_0))$ is automatically $\calD_B(\omega)$-good.

A key property of these good and bad cubes is that under a random choice of  $\omega \in \Omega$  a cube $R \in \calD_{0}$ has a small probability of being  $\calD_B(\omega)$-bad. The version of this fact that we will use is formulated in the following lemma.   For every $k,l \in \Z, k \leq l$,  define 
$$
\Omega_k^l =\{ \omega \in \Omega: \omega(m)=\omega_0(m) \text{ if } m<k \text{ or } m>l\}.
$$
We equip $\Omega_k^l$ again with the natural probability measure with which the coordinates $\omega(m), k \leq m \leq l$, are independent and uniformly distributed over $\{0, \dots, L\} \times \{1, \dots ,M\}$. Note that this is a finite probability space.

\begin{lem}\label{probability of badness}
There exist two constants $C=C(L,M)>0$ and $ \eta \in (0,1]$ so that the following holds. Fix some big enough (depending on $\gamma$) goodness parameter $r$. Suppose $B \subset E$ is a ball in $E$ and let $\calD_0=\calD(\omega_0)$ and $\calD_B(\omega) \subset \calD(\omega)$ be the dyadic lattices related to this ball as described above.  Assume $k_0 \in \Z$ is such that $ \ell(Q_B(\omega))=\delta^{k_0}$ for some, and hence for every, $\omega \in \Omega$.  Let $k_1 \in \Z$ be any number such that $k_1 \geq k_0+r$. Then, for every cube $R \in \calD_0$ with $\ell(R) \geq \delta^{k_1} $ it holds that
\begin{equation}\label{prob. of bad}
\mathbb{P}\big(\{ \omega \in \Omega^{k_1}_{k_0} : R \text{ is } \calD_B(\omega) \text{-bad}\}\big) \leq C \delta^{\gamma r \eta}.
\end{equation}
\end{lem}

The point in the reduction to these finite spaces $\Omega_{k_0}^{k_1}$ is a certain technical problem related to measurability in the big piece global $Tb$ theorem. If one replaces $\Omega_{k_0}^{k_1}$ with  $\Omega$ in \eqref{prob. of bad}, then the inequality would follow from \cite{HM}, Theorem 10.2. In a similar way as in \cite{HM}  Inequality \eqref{prob. of bad} is also essentially proved in \cite{AH}, Theorem 2.11.  In Appendix \ref{appendix} we sketch the proof of Lemma \ref{probability of badness} just by repeating arguments in \cite{AH} and noticing that it is enough to use $\Omega_{k_0}^{k_1}$ instead of the whole $\Omega$.

\section{Proof of the Main Theorem}\label{proof of main}

Assuming the non-homogeneous good lambda method (Theorem \ref{thm:goodlambda}) and the big pieces global $Tb$ theorem (Theorem \ref{the big piece Tb}), we give the proof of our main theorem, Theorem \ref{thm:main}, here. The following proposition is the main ingredient:

\begin{prop}\label{prop:main}
Let  $C_1,C_2\geq 1$ and $\varepsilon_0 \in (0,1)$ be given constants. Let $B$ be a  closed ball with radius $r$ in $E$.  Suppose $\mu$ is a measure of order $m$ in $E$ and that there exists a  complex measure $\nu \in M(E)$ such that 
\begin{enumerate}
\item  $ \supp \nu \subset B $;
\item $\nu(B)= \mu(B)$;
\item $|\nu|(B) \leq C_1 \mu(B)$;
\item If $A \subset B$ is a subset so that $\mu(A) \leq \varepsilon_0 \mu(B)$, then $|\nu|(A) \leq \frac{1}{16C_1} |\nu|(B)$.
\end{enumerate}
Assume further that there exists some $s>0$ and a Borel set $U \subset B$ so that 
\begin{enumerate}[a)]
\item $|\nu|(U) \leq \frac{1}{16C_1} |\nu|(B)$;
\item $\sup_{\lambda>0} \lambda ^s \mu(\{y \in B \setminus U: \calC^r \nu(y) >\lambda \}) \leq C_2 |\nu|(B)$.
\end{enumerate}

Then there exists a set $G \subset B\setminus U$ with $\mu(G) > \varepsilon_0 \mu(B)$ so that 
$$
\|1_G \calC_\mu f \|_{L^2(\mu)}  \lesssim  \|f\|_{L^2(\mu)}
$$
holds for all $f \in L^2(\mu)$ with $\{y \in E \colon  f(y) \not=0\} \subset G$.
\end{prop}

\begin{proof}
We may assume that $\mu(E \setminus B)=0$, because once we prove this with such measures then in the general case we may apply it with $\mu \lfloor B$. 

Let $y \in B$. Then
\begin{equation*}
\begin{split}
\calC_{r} \nu(y)  \lesssim |\nu|(B) \Big(\int_{\Gamma_r(y)} d(x,E)^{-2m} \ud \sigma(x)\Big)^{\frac{1}{2}}  \lesssim \frac{|\nu|(B)}{r^m} 
\leq C_1 \frac{\mu(B)}{r^m} \lesssim 1,
\end{split}
\end{equation*}
where we used a similar estimate as in the proof of Lemma \ref{cone computation} in the second step.
Hence 
$$
\sup_{\lambda>0} \lambda^s \mu\big(\{ y \in B \setminus U \colon \calC_r\nu(y) > \lambda\}\big) \lesssim \mu(B) \leq |\nu|(B),
$$
and this combined with the weak type assumption $b)$ shows that 
\begin{equation}\label{untruncated weak}
\sup_{\lambda>0} \lambda ^s \mu(\{y \in B \setminus U: \calC \nu(y) >\lambda \}) \leq C_2' |\nu|(B)
\end{equation}
holds for some constant $C_2'$.

Let $b$ be a function such that $|b(y)|=1$ for all $y \in B$ and $\nu=b  |\nu|$. Note that 
$$
\calC \nu (y) = \calC (b|\nu|)(y)=\calC_{|\nu|}b(y)$$
for all $y \in E$. The idea is to use the big piece global $Tb$ theorem \ref{the big piece Tb} with the measure $|\nu|$ and the bounded function $b$. Hence we have to verify the corresponding assumptions listed in the statement of Theorem \ref{the big piece Tb}. Also, to come back to the measure $\mu$, we will show that $|\nu|$ and $\mu$ are comparable in a big piece of $B$.

To begin, recall the dyadic lattices $\calD_B(\omega)$ related to the ball $B$.  Fix some random parameter $\omega \in \Omega$ and let $\mathcal{A}_\omega$ be the collection of the maximal dyadic cubes $R \in  \calD_B(\omega)$ such that
$$
\Big| \int_R b \ud |\nu| \Big| \leq \eta |\nu|(R),
$$  
where $\eta \in (0,1)$ is a small number to be specified. Set $T_\omega:= \bigcup_{R \in \calA_\omega} R$. First estimate
$$ 
\int_{Q_B(\omega)} b \ud |\nu| = \nu(Q_B(\omega))=\mu(B) \geq \frac{1}{C_1}|\nu|(Q_B(\omega)).
$$
Using this we get
\begin{equation*}
\begin{split}
\frac{1}{C_1}|\nu|(Q_B(\omega)) & \leq  \int_{Q_B(\omega)} b \ud  |\nu| =  \int_{Q_B(\omega) \setminus T_\omega} b \ud |\nu|+ \sum_{R \in \calA_\omega} \int_R b \ud |\nu| \\
& \leq |\nu|(Q_B(\omega) \setminus T_\omega) + \eta |\nu|(T_\omega) \\
&= |\nu|(Q_B(\omega))+(\eta-1)|\nu|(T_\omega),
\end{split}
\end{equation*}
which can be written as 
$$
|\nu|(T_\omega) \leq \frac{C_1-1}{C_1(1-\eta)} |\nu|(Q_B(\omega)).
$$
If $\eta=\eta(C_1)$ is chosen suitably, then 
$$
\frac{C_1-1}{C_1(1-\eta)}=1-\frac{1}{2C_1}<1.
$$

Next, define
$$
H_0:=\{y \in E: M^m \nu(y) >p_0\},
$$
where $p_0>0$ will be fixed. For every $y \in H_0$ set 
$$
r(y):= \sup\Big\{r>0: \frac{|\nu|\big(B(y,r)\big)}{r^m}>p_0\Big\}.
$$
If $y \in H_0$ and $z \in B(y,r(y))$, then 
$$
\frac{|\nu|\big(B(z,2r(y))\big)}{(2r(y))^m} \geq 2^{-m} \frac{|\nu|\big(B(y,r(y))\big)}{r(y)^m} \geq 2^{-m}p_0.
$$
Hence, if we define 
$$
H_1:= \bigcup_{y \in H_0} B(y,r(y)),
$$
we see that $H_1 \subset \{y \in E \colon M^m \nu (y) \geq 2^{-m}p_0\}$. 

Because of the weak $(1,1)$ boundedness of the maximal function and the assumption (3), we have
\begin{equation*}
\mu(H_1) \leq \mu(\{y \in E \colon M^m \nu (y) \geq 2^{-m}p_0\}) \leq C\frac{2^m}{p_0} |\nu|(B)
 \leq C\frac{2^m }{p_0}C_1 \mu(B).
\end{equation*}
Set $p_0 :=C2^m C_1\varepsilon_0^{-1}$, whence the assumption (4) gives $|\nu|(H_1) \leq \frac{1}{16C_1} |\nu|(B)$. 

Now we prove the comparability of the measures $\mu$ and $|\nu|$ in a subset of $B$. Define $\calF_1$ to be the set of maximal cubes $R \in \calD_B(\omega)$ such that
$$
|\nu|(R) \leq \frac{1}{16C_1}\mu(R),
$$
and $\calF_2$ to be the set of maximal cubes $R\in \calD_B(\omega)$ such that
$$
|\nu|(R) \geq \frac{C_1}{\varepsilon_0} \mu(R).
$$
As an immediate consequence of the definition of $\mathcal{F}_1$ we get
\begin{equation*}
\begin{split}
|\nu|\Big(\bigcup_{R \in \calF_1} R\Big) =\sum_{R \in \calF_1}|\nu|(R) &\leq \frac{1}{16C_1}\sum_{R \in \calF_1}\mu(R)  \\
&\leq \frac{1}{16C_1}\mu(Q_B(\omega)) \\
&\leq \frac{1}{16C_1}|\nu|(B).
\end{split} 
\end{equation*}
Also,  it holds that
\begin{equation*}
\begin{split}
\sum_{R \in \calF_2} \mu(R) \leq \sum_{R \in \calF_2} \frac{\varepsilon_0}{C_1}|\nu|(R) \leq \frac{\varepsilon_0}{C_1} |\nu|(B) \leq \varepsilon_0 \mu(B),
\end{split}
\end{equation*}
and accordingly
$$
|\nu|\Big(\bigcup_{R \in \calF_2}R\Big) \leq \frac{1}{16C_1}|\nu|(B)
$$ 
by assumption (4) again. Hence the set 
$$
H_2:=\bigcup_{R \in \calF_1 \cup  \calF_2} R
$$ 
satisfies
$$
|\nu|(H_2) \leq \frac{1}{8C_1}|\nu|(B).
$$  

If $Q \in \calD_B(\omega)$ is such that $Q  \not \subset H_2$, then
\begin{equation}\label{measures comparable}
\frac{1}{16C_1}\mu (Q) \leq |\nu| (Q) \leq \frac{C_1}{\varepsilon_0}\mu (Q).
\end{equation}
From this we can conclude (using a dyadic variant of Lemma 2.13 of \cite{MaBook}) that for all Borel sets $A \subset \R^n$ there holds that
$$
\frac{1}{16C_1} \mu(A \cap (B \setminus H_2)) \le |\nu| (A \cap (B \setminus H_2)) \le \frac{C_1}{\varepsilon_0} \mu(A \cap (B \setminus H_2)).
$$
In particular, we have that $|\nu| \rest (B \setminus H_2) \ll \mu \rest (B \setminus H_2)$.  Radon--Nikodym theorem gives us a Borel function $\varphi \ge 0$ so that
$$
|\nu|(A) = \int_A \varphi \ud \mu
$$
and 
\begin{equation}\label{phi sim 1}
\frac{1}{16C_1} \leq \varphi (y) \leq \frac{C_1}{\varepsilon_0}
\end{equation}
hold for all Borel sets $A \subset B \setminus H_2$ and $\mu$-a.e. $y \in B \setminus H_2$.

Now we have constructed all the necessary sets. Define $H:=H_1 \cup H_2 \cup U$. If $|\nu|(B(y,r)) > p_0 r^m$, then $B(y,r) \subset H$, and 
\begin{equation*}
\begin{split}
|\nu|(T_\omega \cup H) &\leq |\nu|(T_\omega)+|\nu|(H_1) \cup |\nu|(H_2) \cup |\nu|(U) \\ & \leq \Big(1-\frac{1}{2C_1}+\frac{1}{16C_1} + \frac{1}{8C_1} +\frac{1}{16C_1}\Big)|\nu|(B)=\Big(1-\frac{1}{4C_1}\Big)|\nu|(B).
\end{split}
\end{equation*}
Furthermore, the weak type condition \eqref{untruncated weak} and Equation \eqref{phi sim 1}  give
\begin{equation*}
\begin{split}
\sup_{\lambda>0} \lambda^s |\nu|\big(\{y \in B\setminus H \colon \calC_{|\nu|}b> \lambda\}\big) 
&\lesssim \sup_{\lambda>0} \lambda^s \mu\big( \{y \in B\setminus U \colon  \calC \nu(y)> \lambda\}\big) \\
& \leq C_2' |\nu|(B).
\end{split}
\end{equation*}

Since the assumptions of the big piece global $Tb$ theorem \ref{the big piece Tb} are verified,  we may apply it to give a set $G \subset B \setminus H$ with
\begin{equation}\label{size of G}
|\nu|(G) \geq \frac{1-(1-\frac{1}{4C_1})}{3} |\nu|(B)= \frac{1}{12C_1}|\nu|(B)
\end{equation}
such that
$$
\|1_G \calC_{|\nu|} f \|_{L^2(|\nu|)} \lesssim_{C_1,C_2,\varepsilon_0} \|f\|_{L^2(|\nu|)}
$$
holds for every $f \in L^2(|\nu|)$. Note that it must be that $\mu(G)> \varepsilon_0 \mu(B)$, because otherwise \eqref{size of G} would be contradicted by the assumption (4).

Suppose now $f \in L^2(\mu)$ with $\{y \in E \colon f(y) \not= 0\} \subset G$ and define the function $g:= f/\varphi$, which is understood to be zero in  $\{y \in E \colon f(y)= 0\}$.  Remember that $|\nu|\lfloor (B \setminus H_2)=\varphi \mu\lfloor (B \setminus H_2)$ and $\varphi \sim 1$ for $|\nu|$- and $\mu$- a.e. $y \in B \setminus H$. Then
\begin{equation*}
\begin{split}
\|1_G \calC_\mu f\|_{L^2(\mu)} = \|1_G \calC_{|\nu|} g\|_{L^2(\mu)} \sim \| 1_G \calC_{|\nu|} g \|_{L^2(|\nu|)}& \lesssim \|g \|_{L^2(|\nu|)} \\
& =\| f/\varphi\|_{L^2(|\nu|)} \\ &\sim \|f\|_{L^2(\mu)}. 
\end{split}
\end{equation*}
 
This concludes the proof.

\end{proof}

With this proposition we can easily  prove the main theorem:
\begin{proof}[Proof of Theorem \ref{thm:main}]
Let $B$ be a closed $(10, b)$-doubling ball in $E$ with a $\kappa$-small boundary. Then there exists by assumption a measure $\nu_B$  related to the ball $B$ as in Proposition \ref{prop:main}. Thus, an application of that proposition gives a set $G_B \subset B$ with $\mu(G_B) > \varepsilon_0 \mu(B)$ such that
$$
\| 1_{G_B} \calC_\mu f \|_{L^2(\mu)} \lesssim \| f \|_{L^2(\mu)}
$$
holds for all $f \in L^2(\mu)$ with $\{y \in E \colon f(y) \not=0\} \subset G_B$. 

Since this happens in every closed $(10,b)$-doubling ball with a $\kappa$-small boundary, the good lambda theorem \ref{thm:goodlambda} and Remark \ref{L^2 implies weak (1,1)} following it imply that $\calC_\mu$ is bounded in $L^p(\mu)$ for every $p \in (1,\infty)$.
\end{proof}

\section{A geometric problem related to cones}\label{geometric problem}
Before going to the good lambda method and the big pieces global $Tb$ theorem, we consider a certain geometric problem related to cones. Namely, we want to estimate the $\sigma$-measure of two truncated cones that are close to each other. Compared to the upper half-space case, the difficulty here is that a cone defined with respect to a general set $E$ does not have such a simple form. This estimate will be needed to have  certain continuity for a truncated square function.  

\begin{lem}\label{lem:geometric problem} 
Let $t \geq 10$. Then for all $r>0$ and $y,y' \in E$ with $|y-y'| <r$ it holds that
\begin{equation}\label{eq:geometric problem}
\sigma\big(\Gamma_{tr}(y) \Delta \Gamma_{tr}(y')\big) \lesssim  \frac{1}{t},
\end{equation}
where $\Gamma_{tr}(y) \Delta \Gamma_{tr}(y'):= \Gamma_{tr}(y) \setminus \Gamma_{tr}(y') \cup \Gamma_{tr}(y') \setminus \Gamma_{tr}(y)$.
\end{lem}

\begin{proof}
Let  $x \in \Gamma_{tr}(y) \setminus \Gamma_{tr}(y')$. The crucial observation is that then 
\begin{equation}\label{distance to E}
|x-y| < 2 d(x,E) \leq |x-y'| < |x-y| +r,
\end{equation}
so $x$ is at a quite specific distance to $E$. Hence there exists $\ty_x \in E$ so that 
\begin{equation*}%\label{distance to some y}
|x-\tilde{y}_x| < \frac{|x-y|+r}{2},
\end{equation*}
and for all $\ty \in E$ there holds that
$$
|x-\ty|>\frac{|x-y|}{2}.
$$
For every $\tilde{y} \in E$ define the sets
\begin{equation}\label{definition of B}
\bar{B}_{\tilde{y}}:= \Big\{x \in \R^n : |x-\tilde{y}| \leq  \frac{|x-y|}{2}\Big\} 
\end{equation}
($\bar{B}$ indicates that it will turn out to be a closed ball) and 
\begin{equation}\label{definition of G}	
G_{\tilde{y}}:=\Big\{x \in \R^n : |x-\tilde{y}| <\frac{|x-y|+r}{2}\Big\}.
\end{equation}
The above considerations show that
$$
\Gamma_{tr}(y) \setminus \Gamma_{tr}(y') \subset \Big[\bigcup_{\tilde{y} \in E} G_{\tilde{y}} \setminus \bigcup_{\tilde{y} \in E} \bar{B}_{\tilde{y}}\Big] \cap \Gamma_{tr}(y).
$$

Since every $G_{\tilde{y}}$ is open, there is a countable collection $\{y_i\}_{i \in \tilde{\calI}} \subset E$ so that $\bigcup_{i \in \tilde{\calI}}G_{y_i}=\bigcup_{\ty \in E}G_{\ty}$. Then, we can find a finite subcollection $\calI \subset \tilde{\calI}$ so that 
\begin{align*}
\sigma\Big(\Big[\bigcup_{\tilde{y} \in E} G_{\tilde{y}} \setminus \bigcup_{\tilde{y} \in E} \bar{B}_{\tilde{y}}\Big] \cap \Gamma_{tr}(y)\Big) 
&\leq 2\sigma\Big(\Big[\bigcup_{i \in \calI} G_{y_i} \setminus \bigcup_{\tilde{y} \in E} \bar{B}_{\tilde{y}}\Big] \cap \Gamma_{tr}(y)\Big) \\
&\leq 2 \sigma\Big(\Big[\bigcup_{i \in \calI} G_{y_i} \setminus \bigcup_{i \in \calI} \bar{B}_{y_i}\Big] \cap \Gamma_{tr}(y) \Big).
\end{align*}
We can clearly assume that $y_i \not= y_j$, for all $i,j \in \calI, i \not=j$. For every $k \in \{0,1,2,\dots\}$ let $\calI_k \subset \calI$ be the set of those indices $i$ such that $G_{y_i} \cap \Gamma_{2^ktr}^{2^{k+1}tr}(y)\not=\emptyset$, whence
\begin{equation}\label{Whitney sections}
\begin{split}
\sigma\Big(\Big[\bigcup_{i \in \calI} G_{y_i} \setminus \bigcup_{i \in \calI} \bar{B}_{y_i}\Big] \cap \Gamma_{tr}(y)  \Big)
& = \sum_{k=0}^\infty \sigma\Big( \Big[\bigcup_{i \in \calI_k} G_{y_i} \setminus \bigcup_{i \in \calI_k} \bar{B}_{y_i}\Big] \cap \Gamma_{2^ktr}^{2^{k+1}tr}(y) \Big).
\end{split}
\end{equation}
Now we fix some $k$ for the rest of the proof and show that
\begin{equation}\label{fixed k}
m_n\Big(\bigcup_{i \in \calI_k} G_{y_i} \setminus \bigcup_{i \in \calI_k} \bar{B}_{y_i}\Big) \lesssim (2^ktr)^{n-1}r,
\end{equation}
where $m_n$ is the $n$-dimensional Lebesgue measure. Using this we may infer from \eqref{Whitney sections} that
$$
\sigma\Big(\Big[\bigcup_{i \in \calI} G_{y_i} \setminus \bigcup_{i \in \calI} \bar{B}_{y_i}\Big] \cap \Gamma_{tr}(y) \Big) \lesssim \sum_{k=0}^\infty \frac{(2^ktr)^{n-1}r}{(2^ktr)^n} \sim \frac{1}{t},
$$
which proves the lemma.

\subsection*{$\bar{B}_{y_i}$ is a ball} Let $i \in \calI_k$. First we'll show that $\bar{B}_{y_i}$ is a closed ball. We write a general point $x \in \R^n$ with coordinates as $x=(x(1), \dots, x(n))$. The condition $|x-y_i|\leq\frac{|x-y|}{2}$ can be written as
$$
\sum_{m=1}^n\big(x(m)-y_i(m)\big)^2 \leq \frac{1}{4} \sum_{m=1}^n \big(x(m)-y(m)\big)^2,
$$
and further as
$$
\sum_{m=1}^n\big(x(m)-(\frac{4}{3}y_i(m)-\frac{1}{3}y(m))\big)^2 \leq \frac{4}{9}\sum_{m=1}^n \big(y(m)-y_i(m)\big)^2.
$$
From here we see that
\begin{equation}\label{ball}
\bar{B}_{y_i}= \bar{B}\big(\frac{4}{3}y_i-\frac{1}{3}y, \frac{2}{3}|y-y_i|\big).
\end{equation}
The centers of the balls $\bB_{y_i}$ and $\bB_{y_j}$ are different if $ i \not= j$. 

Consider still the fixed $i \in \calI_k$. By definition there exists a point $x \in G_{y_i}\cap \Gamma_{2^ktr}^{2^{k+1}tr}(y)$. The distance $|y-y_i|$ can be estimated as 
$$
|y-y_i| \leq |y-x| +|x-y_i| < |y-x|+ \frac{|y-x|+r}{2} \leq 2|x-y|,
$$
because $|y-x| >10r$. On the other hand
$$
|y-y_i| \geq |x-y|-|x-y_i| > |x-y|- \frac{|x-y|+r}{2} \geq \frac{|x-y|}{3}.
$$
Combining these we get 
\begin{equation}\label{radius}
r(\bar{B}_{y_i})=\frac{2}{3}|y- y_i| \sim |x-y| \sim 2^k tr.
\end{equation}
Also, from \eqref{radius} and the fact
$$
\Big|\Big(\frac{4}{3}y_i-\frac{1}{3}y\Big) - y\Big| = \frac{4}{3} |y-y_i| \sim 2^k tr
$$
 it follows that there exists an absolute constant $C$ such that
\begin{equation}\label{contained}
\bar{B}_{y_i} \subset B(y,C2^ktr).
\end{equation}

\subsection*{}

Re-index the balls $\{\bB_{y_i}\}_{ i \in \calI_k}$ as $\bar{B}_1=\bB(x_1,r_1),\dots,\bar{B}_K =\bB(x_K,r_K)$ for some $K \in \N$, and write correspondingly $G_1,\dots,G_K$. 
Then, set 
$$
A_1:=\partial \Big(\bigcup_{j=1}^K \bB_j\Big) \cap \partial \bar{B}_{1},
$$
and for every $i \in \{2,\dots,K\}$ define
$$
A_i:= \Big\{x \in \partial \Big(\bigcup_{j=1}^K \bB_{j}\Big)\colon x \in \partial \bar{B}_i \setminus \Big( \bigcup_{j=1}^{i-1} \partial \bB_{j}\Big)\Big\}.
$$ 
The sets $A_i$ are pairwise disjoint and $\bigcup_{i=1}^K A_i=\partial\Big( \bigcup_{j =1}^K \bB_{j}\Big)$. With the sets $A_i$ we still define
$$
V_i=\big\{x \in \R^n: x=(1-\lambda)x_i+\lambda a \text{ for  some } a \in A_i \text{ and } \lambda \in [0,1]\big\}.
$$ 
In other words, $V_i$ is the set of points that are on a segment whose one end is $x_i$ and the other is on $A_i$. 

\subsection*{The sets $V_i$ are pairwise disjoint} We claim that $V_i \cap V_j= \emptyset $ if $i \not = j$. To get a contradiction, suppose that there exist $a_i \in A_i$, $a_j \in A_j$ and $\lambda_i$, $\lambda_j \in [0,1]$ such that 
\begin{equation}\label{eq:xconvex-i=j}
x=(1-\lambda_i)x_i+\lambda_i a_i=(1-\lambda_j)x_j+\lambda_j a_j,
\end{equation}
where $i,j \in \{1,\dots,K\}$ and $i\not=j$.  
Assume first $|x-a_j|>|x-a_i|$ and notice that by  \eqref{eq:xconvex-i=j},
$$
a_i \in B(x, |x-a_j|) \subset B(x_j,|x_j-a_j|) \subset \text{int} \Big( \bigcup_{l=1}^K \bB_l \Big),
$$
which is a contradiction because $a_i$ is supposed to be on the boundary. The case $|x-a_i| >|x-a_j|$ is handled similarly. Thus, we can only have that  
\begin{equation}\label{equal distance}
|x-a_i|=|x-a_j|.
\end{equation}

Without loss of generality assume that
\begin{equation}\label{B_i bigger}
|x_i-a_i| \geq |x_j-a_j|. 
\end{equation}
Suppose first $\lambda_i\not=0$. Then
$$
a_j \in \bB(x,|x-a_j|)=\bB(x,|x-a_i|) \subset B(x_i,|x_i-a_i|) \cup \{a_i\},
$$
which is a contradiction. Indeed, we have $a_j \not  \in B(x_i,|x_i-a_i|)$ because $a_j$ is a boundary point, and $a_i \not= a_j$ since the sets $A_i$ are pairwise disjoint.  

Suppose then $\lambda_i=0$, which implies that $x=x_i$. This combined with \eqref{eq:xconvex-i=j}, \eqref{equal distance} and \eqref{B_i bigger}  implies that $\lambda_j=0$, since
$$
|x-a_j| = |x-a_i| = |x_i - a_i| \ge |x_j-a_j|.
$$
This gives $x_j=x=x_i$. This is again a contradiction because $x_i \not= x_j$ (as noted after \eqref{ball}). Thus we have shown that $V_i \cap V_j = \emptyset$ if $i \not= j$.

\subsection*{Proof of \eqref{fixed k}}
Let $ x \in \bigcup_{j=1}^K G_j \setminus \bigcup_{j=1}^K\bB_j$, and suppose $i \in \{1, \dots, K\}$ and $a \in A_i$ are  such that 
\begin{equation} \label{distance to boundary}
|x-a|=d\Big(x,\bigcup_{j=1}^K A_j\Big)= d\Big(x,\bigcup_{j=1}^K \bB_j\Big).
\end{equation}
Then, because $a$ minimizes the distance of $x$ to the ball $\bB(x_i,r_i)$, $x$ has to be on the same line with $a$ and $x_i$. Otherwise there would be a point $x' \in \bB_i$ with $|x-x'| < |x-a|$, which contradicts \eqref{distance to boundary}. From the definitions \eqref{definition of B} and \eqref{definition of G} of the sets $\bB_i$ and $G_i$ it follows that $|x-a| \leq r$. Thus
\begin{equation}\label{close to boundary}
\bigcup_{i=1}^K G_i \setminus \bigcup_{i=1}^K \bB_i \subset \bigcup_{i=1}^K F_i,
\end{equation}
where
$$
F_i := \big\{x_i +t(a-x_i) \colon  t \in \big(1, 1+r/r_i\big], a \in A_i\big\}.
$$

Fix some $i$ for the moment and recall the set $V_i$ from above. We want to compare the Lebesgue measures of $F_i$ and $V_i$.  Note that the set $(V_i-x_i) \setminus \{\bar{0}\}$ can be written as a disjoint union
$$
(V_i-x_i)\setminus \{\bar{0}\} = \bigcup_{k=1}^\infty \Big( \frac{r_i}{r_i+r} \Big)^k (F_i-x_i).
$$
Hence
\begin{equation*}
m_n(V_i)= \sum_{k=1}^\infty \Big( \frac{r_i}{r_i+r} \Big)^{nk} m_n(F_i)=  \frac{r_i^n}{(r_i+r)^n-r_i^n} m_n(F_i).
\end{equation*}
Using the mean value theorem and the fact that $r_i \sim 2^ktr$ we have
$$
\frac{r_i^n}{(r_i+r)^n-r_i^n} \sim  \frac{r_i}{r} \sim 2^kt,  
$$
and thus 
$m_n(V_i) \sim 2^kt \cdot m_n(F_i)$.

Remember that the sets $V_i \subset \bB_i$ are pairwise disjoint and that $\bB _i \subset B(y,C2^ktr)$ for every $i \in \{1, \dots, K\}$, as stated in \eqref{contained}. Now we can estimate

\begin{equation*}
\begin{split}
\sum_{i=1}^K m_n(F_i) \sim \frac{1}{2^kt} \sum_{i=1}^K m_n(V_i)  
\leq \frac{m_n(B(y,C2^ktr))}{2^kt} \sim  (2^ktr)^{n-1}r,
\end{split}
\end{equation*}
which in view of \eqref{close to boundary} completes the proof of \eqref{fixed k}. The proof of Lemma \ref{lem:geometric problem} is complete.

\end{proof}

\section{The non-homogeneous good lambda method}\label{sec:good lambda}
In this section we prove the non-homogeneous good lambda method of Tolsa \cite{ToBook} in our setting. For this, we shall need the geometric considerations from Section \ref{geometric problem}.

In the proof of the good lambda inequality we shall use the following Whitney type argument, Lemma \ref{Whitney type lemma}, which
allows the use of regular balls only. This is  a version of Lemma 2.23 in \cite{ToBook} adapted to our situation, with some additional arguments from \cite{MMT} related to small boundaries and
the usage of balls.

First, we record the following fact from \cite{NTV} (see also \cite{ToBook} and \cite{V}).
\begin{lem}\label{lem:small boundary}
Let $\mu$ be a Radon measure in $\R^n$ and let $\kappa$ be a big enough constant depending only on the dimension $n$. Suppose $B(x,r)$ is a ball (open or closed) in $\R^n$. Then there exists $R \in [r, 1.2r]$ so that the ball $B(x,R)$ has a $\kappa$-small boundary.
\end{lem}

If $B=B(x,r)$ is a ball in $E$ or in $\R^n$ and $s>0$, we define $sB:=sB(x,r):=B(x,sr)$, and similarly with closed balls. 

\begin{lem}\label{Whitney type lemma}
Let $\mu$ be a Borel measure with $\supp \mu \subset E$. Suppose $ U \subsetneq E$ is a relatively open set with $\mu(U)<\infty$. Assume $a \geq 3, \rho \geq 16 a$ and let $b$ be a big enough constant depending on $\rho$ and the dimension $n$. Let also $\kappa$  be a big enough constant depending only on the dimension $n$. Recall the constant $\delta=\frac{1}{1000}$ related to dyadic lattices in $E$.  Define $C_1:= \frac{\rho}{8}\geq 6$ and $C_2:= \frac{(12+\rho)\delta^{-1}}{6}\geq 10000$. Then there exist a constant $D_0=D_0(\rho,n)$ and a finite collection of closed balls $\{B_i\}_{i \in \mathcal{I}}$  \emph{in $E$}  with the following properties:
\begin{itemize}
\item $B_i \cap B_j = \emptyset$ if $i \not= j$.
\item For every $i \in \calI$ there exist at most $D_0$ indices $j \in \calI$ so that $\frac{C_1}{2}B_i \cap \frac{C_1}{2}B_j \not= \emptyset$.
\item For every $i \in \mathcal{I}$ it holds that $C_1 B_i \subset U$ and $C_2 B_i \cap  (E \setminus U) \not= \emptyset$.
\item The balls $B_i$ are $(a,b)$-doubling and have $\kappa$-small boundary.
\item $\mu \big(\bigcup_{i \in \mathcal{I}}B_i\big) \geq \frac{1}{2b}\mu(U)$.
\end{itemize}
\end{lem}

\begin{proof}
Let $\calD$ be any dyadic lattice in $E$ as described in Section \ref{preliminaries}. Consider the maximal dyadic cubes $Q \in \calD$ such that
\begin{equation}\label{Whitney cubes}
d(Q, E\setminus U) \geq \rho \ell(Q).
\end{equation}
That $Q$ is a maximal cube such that \eqref{Whitney cubes} holds means that there does not exist a cube $R \in \calD$ satisfying \eqref{Whitney cubes} so that $R \supset Q$ and $\ell(R) >\ell(Q)$. Let $\{Q_i\}_{i \in \mathcal{K}} \subset \calD$ be the collection of these maximal cubes. Then $U =\bigcup_{i \in \mathcal{K}}Q_i$ and the cubes in $\{Q_i\}_{i \in \mathcal{K}}$ are pairwise disjoint.

Suppose $i \in \mathcal{K}$. Recall that if $Q \in \calD$ then $\widehat{Q} \in \calD$ is the unique cube with $\ell(\widehat{Q}) = \delta^{-1}\ell(Q)$  that contains $Q$. By construction we know that $d(\widehat{Q_i},E \setminus U) < \rho \ell(\widehat{Q_i})$. Hence 
\begin{equation*}
\begin{split}
d(c_{Q_i},E \setminus U) &\leq \diam( \widehat{Q_i}) +d(\widehat{Q_i},E \setminus U) \\
& < 12 \ell(\widehat{Q_i}) +\rho \ell(\widehat{Q_i}) \\
&=(12+\rho) \delta^{-1}\ell(Q_i)=C(\rho) \ell(Q_i),
\end{split}
\end{equation*}
where $C(\rho):= (12+\rho) \delta^{-1}$. Hence, it holds for all $i \in \mathcal{K}$ that
\begin{equation}\label{distance to complement}
\rho \ell(Q_i) \leq d(c_{Q_i},E \setminus U) < C(\rho)\ell(Q_i).
\end{equation}

Next we prove the existence of the constant $D_0$. Suppose $i,j \in \mathcal{K}$ so that 
$$
\bB_E(c_{Q_i},\frac{\rho}{2}\ell(Q_i))\cap \bB_E(c_{Q_j},\frac{\rho}{2}\ell(Q_j)) \not = \emptyset,
$$ 
and suppose $\ell(Q_j)=\delta^k\ell(Q_i)$ for some $k \in \Z, k \geq 0$. Then 
\begin{equation*}
\begin{split}
d(c_{Q_i},E \setminus U) &\leq |c_{Q_i}-c_{Q_j}| + d(c_{Q_j},E \setminus U) \\
&\leq \frac{\rho}{2}\ell(Q_i) + \frac{\rho}{2}\ell(Q_j) + C(\rho)\ell(Q_j) \\
& =(\frac{\rho}{2}+ \frac{\rho}{2}\delta^k+C(\rho) \delta^k) \ell(Q_i).
\end{split}
\end{equation*}
Thus, because of \eqref{distance to complement}, we see that there exists $k_0 \in \Z$ depending on $\rho$ such that $k \leq k_0$.

Fix now some ball $\bB_E(c_{Q_i},\frac{\rho}{2}\ell(Q_i))$ and let $\mathcal{K}_i$ be the set of those indices $j$ such that $\bB_E(c_{Q_i},\frac{\rho}{2}\ell(Q_i))\cap \bB_E(c_{Q_j},\frac{\rho}{2}\ell(Q_j)) \not = \emptyset$. Then for all $j\in \mathcal{K}_i$ it holds that 
$$ 
\delta^{k_0}\ell(Q_j) \leq \ell(Q_i) \leq \delta^{-k_0}\ell(Q_j).
$$ 
Hence 
\begin{equation}\label{bounded overlap 1}  
c_{Q_j} \in \bB_E\big(c_{Q_i}, \frac{\rho}{2}(1+\delta^{-k_0})\ell(Q_i)\big) \ \ \text{ for all } j \in \mathcal{K}_i.
\end{equation} 
Also, if $j,j' \in \mathcal{K}_i, j \not= j'$, then because the cubes $Q_i, i \in \mathcal{K}$, are pairwise disjoint, we have 
\begin{equation}\label{bounded overlap 2}
|c_{Q_j}-c_{Q_{j'}}| \geq \max \big(\frac{1}{6}\ell(Q_j) ,\frac{1}{6}\ell(Q_{j'}) \big)
\geq \frac{1}{6}\delta^{k_0}\ell(Q_i).
\end{equation}
Equations \eqref{bounded overlap 1} and \eqref{bounded overlap 2} combined imply that  the number of indices in $\calK_i$ is bounded by a constant $D_0$ that depends only on $\rho$ and $n$.

\subsection*{}

Now we start forming the collection we are after. Suppose $i \in \mathcal{K}$ and consider the ball $\bB_E(c_{Q_i}, 6\ell(Q_i))$. Let $B_i:=\bB_E(c_{Q_i}, r_i)$ be a ball with a $\kappa$-small boundary and radius $r_i \in [6\ell(Q_i), 1.2 \cdot 6 \ell(Q_i)]$ given by Lemma \ref{lem:small boundary}. Since $Q_i \subset B_i \subset U$ we have 
$$
U = \bigcup_{i \in \mathcal{K}} B_i.
$$ 
Also, since $1.2\cdot6  \cdot \frac{C_1}{2} = 1.2\cdot 6 \cdot \frac{\rho}{16} \leq \frac{\rho}{2}$, for every $i \in \calI$ there exist at most $D_0$ indices $j \in \calI$ so that $\frac{C_1}{2}B_i \cap \frac{C_1}{2}B_j \not= \emptyset$.

Let $\mathcal{S}\subset \mathcal{K}$ be the set of indices such that the balls $B_i$ are $(a,b)$-doubling with respect to $\mu$. Then, since $1.2\cdot6a \leq \frac{\rho}{2}$, we have
\begin{equation*}
\begin{split}
\mu(\bigcup_{i \in \mathcal{K} \setminus \calS}B_i) & \leq \sum_{i \in \mathcal{K} \setminus \calS}\mu(B_i) \\
&\leq b^{-1}\sum_{i \in \calI \setminus \calS}\mu(a B_i) \\
&\leq \frac{D_0}{b} \mu(U).
\end{split}
\end{equation*}
So, if $b$ is big enough, then 
$$
\mu\Big( \bigcup_{i \in \calS} B_i \Big) \geq \frac{2}{3} \mu(U),
$$
and choosing a sufficiently big finite subcolletion $\calS_1 \subset \calS$, we get
\begin{equation*}
\mu\Big( \bigcup_{i \in \calS_1} B_i \Big) \geq \frac{1}{2} \mu(U).
\end{equation*}

Finally, using the $3r$-covering theorem, choose a subcollection $\calI \subset \calS_1$ so that the balls $B_i, i\in \mathcal{I},$ are pairwise disjoint and 
$$
\bigcup_{i \in \calS_1} B_i \subset \bigcup_{i \in \calI} 3B_i.
$$
Then, since $a \geq 3$, we have
\begin{equation*}
\begin{split}
\sum_{i \in \calI} \mu(B_i) & \geq b^{-1} \sum_{i \in \calI} \mu(3B_i) \\
& \geq b ^{-1} \mu\Big(\bigcup_{i \in \calS_1}B_i\Big) \\
&\geq \frac{1}{2b}\mu(U).
\end{split}
\end{equation*}

The collection $\{B_i\}_{i \in \calI}$ satisfies all the desired properties.

\end{proof}
 
\begin{thm}\label{thm:goodlambda}
Let $\mu$ be a measure of order $m$ in $E$. Let also $b, \kappa>0$ be  big enough constants depending only on $n$, and assume $\theta \in (0,1)$. Suppose for each closed $(10,b)$-doubling  ball $B$  in $E$ with a $\kappa$-small boundary there exists a subset $G_B \subset B$  with
$\mu(G_B)\geq \theta \mu(B)$ so that $\calC \colon \calM(E) \to L^{1,\infty}(\mu \rest G_B)$ is bounded with a uniform constant independent of $B$.
Then $\calC_{\mu}$ is bounded in $L^p(\mu)$ for all $p\in (1,\infty)$ with a constant depending on $p$ and the preceding constants.
\end{thm} 
 
 \begin{rem}\label{L^2 implies weak (1,1)}
In Theorem \ref{thm:goodlambda} the assumption that $\calC \colon \calM(E) \to L^{1,\infty}(\mu \rest G_B)$ is bounded can be replaced by the assumption that $\calC_\mu \colon L^2(\mu \lfloor G_B) \to L^2(\mu \lfloor G_B)$ is bounded, because the latter implies the former using standard reasoning. This is proved in Appendix \ref{weak (1,1)-boundedness}.  
 \end{rem}

\begin{proof}[Proof of Theorem \ref{thm:goodlambda}] 

Fix an exponent $p \in (1,\infty)$.  Since for any $f \in L^p(\mu)$ it holds that
$$
\calC_{\mu,s}^t f(y)=\Big(\int_{\Gamma_s^t(y)} |T_\mu f(x)|^{2} d(x,E)^{2\alpha} \ud \sigma(x) \Big)^{\frac{1}{2}} \nearrow \calC_{\mu}f(y),
$$
as $s \to 0$ and $t \to \infty$, it is enough to bound the operators $\calC_{\mu,s}^t$ uniformly for $s \in (0,1)$ and $t>1$. 

Note that every $\calC_{\mu,s}^t$ is \emph{a priori} bounded in $L^p(\mu)$ by Lemma \ref{truncated is bounded}. Hence, it suffices to prove that
$$
\|\calC_{\mu,s}^t f\|_{L^p(\mu)} 
\lesssim
\|f\|_{L^{p}(\mu)}
$$
holds uniformly for bounded and boundedly supported functions $f$. From now on such a function $f$ is fixed.

Note that the mapping
$$
y \mapsto \calC_{\mu,s}^t f(y), \quad y \in E,
$$
is continuous, and hence the sets $\{\calC_{\mu,s}^t f>\lambda\}$ are open in $E$ for every $\lambda>0$. Indeed, if $y,y' \in E$, then
\begin{equation}\label{continous square}
\begin{split}
|\calC_{\mu,s}^tf(y)&-\calC_{\mu,s}^t f(y')| \\ &\leq \Big(\int_{\R^n \setminus E}|1_{\Gamma_s^t(y)} T_\mu f(x)- 1_{\Gamma_s^t(y')} T_\mu f(x)|^2d(x,E)^{2\alpha} \ud  \sigma(x)\Big)^{\frac{1}{2}} \\
& \lesssim \frac{\|f\|_{L^1(\mu)}}{s^m} \sigma\big(\Gamma_s^t(y) \Delta \Gamma_s^t(y')\big)^{\frac{1}{2}}.
\end{split}
\end{equation}
When $|y-y'|$ is so small that $10 |y-y'| < s$, then from \eqref{eq:geometric problem} it follows that
$$
\sigma\big(\Gamma_s^t(y) \Delta \Gamma_s^t(y')\big) 
\leq \sigma\big(\Gamma_s(y) \Delta \Gamma_s(y')\big) 
 \lesssim \frac{|y-y'|}{s},
$$
which converges to zero as $y' \to y$. 

The main thing to prove in the good lambda method is the \emph{good lambda inequality} \eqref{eq:good lambda}.  It says that for every $\varepsilon>0$ there exists $\delta>0$ such that the following holds. If $\lambda>0$, then
\begin{equation}\label{eq:good lambda}
\mu\big(\{y \in E: \calC_{\mu,s}^t f(y) >(1+\varepsilon)\lambda ;  M_\mu f(y) \leq \delta\lambda\}\big) \leq \Big(1-\frac{\theta}{4b}\Big)\mu\big(\{\calC_{s,\mu}^t f >\lambda\}\big).
\end{equation}
That $\calC_{\mu,s}^t$ is bounded follows from this inequality and the boundedness of $M_\mu$ in a standard manner.

Fix $\lambda,\varepsilon>0$ and let $\delta>0$ be some number to be specified during the proof. We would like to begin the proof of \eqref{eq:good lambda} by applying the Whitney type lemma \ref{Whitney type lemma} to the relatively open set $\{\calC_{\mu,s}^tf>\lambda\}=:\Omega_\lambda$, and for this reason we need that $\Omega_\lambda \subsetneq E$.

Suppose $\diam (E) =\infty$. If $y \not \in \supp f$, then
\begin{equation*}
\begin{split}
\calC_{\mu,s}^t f(y)^2 = \int_{\Gamma_s^t(y)} |T_\mu f(x)|^2 d(x,E)^{2\alpha} \ud \sigma(x) 
& \lesssim \frac{\|f\|_{L^1(\mu)}^2}{d(x,\supp f)^{2m}} \sigma\big(\Gamma_s^t(y)\big) \\
&\lesssim_{s,t} \frac{\|f\|_{L^1(\mu)}^2}{d(y,\supp f)^{2m}} \to 0,
\end{split}
\end{equation*}
as $d(y,\supp f) \to \infty$. Hence in this case $\{\calC^t_{\mu,s}f>\lambda\} \subsetneq E$ holds. And actually we see that $\Omega_\lambda$ is always a bounded set.
 
\subsection*{The case $\diam(E)<\infty$} 
 Suppose $\diam(E)<\infty$, whence $\mu(E)<\infty$, and assume also $\Omega_\lambda=E$. To have something to prove in \eqref{eq:good lambda},  we may suppose that the left hand side there is non-zero. Then there exists $y_0 \in E$ such that $M_\mu f(y_0) \leq \delta\lambda$.
By assumption (by the same interpretation as in Remark \ref{rem:findiam})
there exists a set $G_E \subset E$ with $\mu(G_E) \geq \theta \mu(E)$ where $\calC_{\mu,s}^t \colon \calM(E) \to L^{1,\infty}(\mu \lfloor G_E)$ is bounded. Thus
\begin{equation*}
\begin{split}
\mu\big(\{y \in E: \calC_{\mu,s}^t f(y)>(1+\varepsilon)&\lambda; M_\mu f(y)\leq \delta \lambda\} \big)\\
&\leq \mu\big(E \setminus G_E) + \mu\big(\{y \in G_E: \calC_{\mu,s}^t f(y)>(1+\varepsilon)\lambda\}\big) \\
&\leq (1-\theta)\mu(E) + \frac{C}{(1+\varepsilon)\lambda} \|f\|_{L^1(\mu)} \\
&=(1-\theta)\mu(E) + \frac{C\mu(E)}{(1+\varepsilon)\lambda} \frac{\|f\|_{L^1(\mu)}}{\mu(E)} \\
&\leq (1-\theta)\mu(E) + \frac{C\mu(E)}{(1+\varepsilon)\lambda} \delta\lambda \\
&\leq \Big(1-\frac{\theta}{4b}\Big)\mu(E) = \Big(1-\frac{\theta}{4b}\Big)\mu(\Omega_{\lambda})
\end{split}
\end{equation*}
if $\delta>0$ is small enough.

We have shown that if $\diam E < \infty$, then Inequality \eqref{eq:good lambda} holds  for those $\lambda$ such that $\Omega_\lambda=E$.

\subsection*{}
So in any case we may assume that $\Omega_\lambda \subsetneq E$. As we noted above $\Omega_\lambda$ is a bounded set, and thus $\mu(\Omega_\lambda)$ is finite. Apply Lemma \ref{Whitney type lemma} with parameters $a=10, \rho=160, b $ and $\kappa$ to the open set $\Omega_\lambda$. This choice gives $C_1=20$ and $C_2=\frac{86000}{3}$. Let $\{B_i\}_{i \in \calI}$ be the resulting set of balls in $E$, and let $r_i$ be the radius of $B_i$. 

Since the balls $B_i$ are closed, $(10,b)$-doubling and have a $\kappa$-small boundary, there exists for every $i$ a set $G_i \subset B_i$ as in the assumptions. Hence
\begin{equation}\label{GL:reduction to big pieces}
\begin{split}
\mu\big(\{y &\in E\colon \calC_{\mu,s}^t f(y)>(1+\varepsilon)\lambda;M_\mu f(y) \leq \delta \lambda\}\big)  \\
&\leq \mu\big(\Omega_\lambda \setminus \bigcup_{i \in \calI} B_i \big) +\sum_{i \in \calI}\mu\big(B_i\setminus G_i\big) \\
& + \sum_{i \in \calI} \mu\big(\{y \in G_i: \calC_{\mu,s}^tf(y)>(1+\varepsilon)\lambda;M_\mu f(y) \leq \delta\lambda\}\big) \\
&\leq \Big(1-\frac{\theta}{2b}\Big)\mu(\Omega_\lambda)  +\sum_{i \in \calI} \mu\big(\{y \in G_i: \calC_{\mu,s}^tf(y)>(1+\varepsilon)\lambda;M_\mu f(y) \leq \delta\lambda\}\big).
\end{split}
\end{equation}
It remains to consider the last sum above. 

\subsection*{Step I} Fix some $i \in \calI$. Suppose $y \in G_i$ is such that $\calC_{\mu,s}^tf(y)>(1+\varepsilon)\lambda$ and $M_\mu f(y) \leq \delta \lambda$. First we will show that then
\begin{equation*}%\label{goal step I}
\calC_{\mu,s}^t(1_{2B_i}f)(y) >\frac{\varepsilon}{2}\lambda
\end{equation*}
if $\delta(\varepsilon)$ is small enough. Since
$$
\calC_{\mu,s}^t (1_{2B_i}f)(y) \geq \calC_{\mu,s}^tf(y) -\calC_{\mu,s}^t(1_{(2B_i)^c}f)(y),
$$ 
this follows from showing that 
\begin{equation}\label{step I follows}
\calC_{\mu,s}^t(1_{(2B_i)^c}f)(y) \leq \Big(1+\frac{\varepsilon}{2}\Big)\lambda. 
\end{equation}
 
Assume for the moment that $s \leq 20C_2r_i\leq t$. For $x \in \Gamma(y)$ we have by Lemma \eqref{distance to a point in E} that
\begin{equation*}
\begin{split}
|T_\mu (1_{(2B_i)^c}f)(x)| \lesssim \int_{E\setminus 2Bi} \frac{|f(z)|}{|x-z|^{m+\alpha}} \ud \mu(z) 
&\lesssim \int_{E\setminus 2B_i} \frac{|f(z)|}{|y-z|^{m+\alpha}} \ud \mu(z) \\ 
&\lesssim \int_E \frac{|f(z)|}{(r_i + |y-z|)^{m+\alpha}} \ud \mu(z) \\
&\lesssim r_i^{-\alpha} M_\mu f(y) \leq r_i^{-\alpha}\delta \lambda.
\end{split}
\end{equation*} 
Hence 
\begin{equation}\label{est:good lambda lower}
\begin{split}
\calC_{\mu,s}^{20 C_2r_i}(1_{(2B_i)^c}f)(y) 
&\lesssim r_i^{-\alpha} \delta\lambda \Big(\int_{\Gamma^{20 C_2r_i}(y)} d(x,E)^{2\alpha} d \sigma(x)\Big)^{\frac{1}{2}} \\
& \lesssim r_i^{-\alpha}\delta \lambda(20 C_2 r_i)^\alpha 
\sim \delta\lambda,
\end{split}
\end{equation}
where in the second step we used a similar estimate as in the proof of Lemma \ref{cone computation}. 

Because of the Whitney properties of the balls $B_i$ there exists a point $y' \in C_2B_i \cap E \setminus \Omega_\lambda$, whence by definition $\calC_{\mu,s}^tf(y') \leq \lambda$. Thus, we can estimate
\begin{equation}\label{est:good lambda 1}
\begin{split}
\calC_{\mu,s}^t(1_{(2B_i)^c}f)(y) 
&\leq \calC_{\mu,s}^{20 C_2r_i}(1_{(2B_i)^c}f)(y) + \calC_{\mu,20 C_2r_i}^t(1_{(2B_i)^c}f)(y) \\
&\leq C \delta \lambda + \big|\calC_{\mu,20 C_2r_i}^t(1_{(2B_i)^c}f)(y)-\calC_{\mu,20 C_2r_i}^t(1_{(2B_i)^c}f)(y')\big| \\
& \quad \quad +\calC_{\mu,20 C_2r_i}^t(1_{(2B_i)^c}f)(y'),
\end{split}
\end{equation}
where further
\begin{equation}\label{est:good lambda 2}
\begin{split}
\calC_{\mu,20 C_2r_i}^t(1_{(2B_i)^c}f)(y')  
&\leq \calC_{\mu,s}^tf(y') +\calC_{\mu,20C_2r_i}^t(1_{2B_i}f)(y') \\
&\leq \lambda + \calC_{\mu,20C_2r_i}^t(1_{2B_i}f)(y').
\end{split}
\end{equation}

We continue with estimating the difference in \eqref{est:good lambda 1}. Suppose $x \in \Gamma_{20C_2r_i}(y')$. Then because $|y-y'| < 20^{-1} d(x,E)$, every $z \in E$ satisfies
$$
|x-z| \sim d(x,E)+ |y'-z| \sim d(x,E)+ |y-z|.
$$
This gives $|T_\mu (1_{(2B_i)^c}f)(x)| \lesssim d(x,E)^{-\alpha} M_\mu f(y)$ by Lemma \ref{annulus calculation}. Since the same estimate holds for all $x \in \Gamma(y)$, we have
\begin{equation}\label{continuity of truncated}
\begin{split}
|\calC_{\mu,20C_2r_i}^t&(1_{(2B_i)^c}f)(y)-\calC_{\mu,20C_2r_i}^t(1_{(2B_i)^c}f)(y')|  \\
&\leq \Big( \int_{\Gamma_{20 C_2 r_i}(y) \Delta \Gamma_{20 C_2 r_i}(y')} |T_\mu (1_{(2B_i)^c}f)(x) |^2 d(x,E)^{2\alpha} \ud \sigma(x) \Big)^{\frac{1}{2}} \\
&\lesssim \delta \lambda \sigma \Big(\Gamma_{20 C_2 r_i}(y) \Delta \Gamma_{20 C_2 r_i}(y') \Big)^{\frac{1}{2}} \\
&\lesssim \delta \lambda,
\end{split}
\end{equation} 
where we used Lemma \ref{lem:geometric problem} to estimate the measure of the symmetric difference of the cones.

Now we take care of the last term in \eqref{est:good lambda 2}. Note that  
$$|T_\mu (1_{2B_i}f)(x)| \lesssim d(x,E)^{-m-\alpha}\|1_{2B_i}f\|_{L^1(\mu)}
$$ 
holds for every $x \in \R^n\setminus E$. Hence
\begin{equation*}
\begin{split}
\calC_{\mu,20C_2r_i}^t(1_{2B_i}f)(y') & \lesssim \|1_{2B_i}f\|_{L^1(\mu)} \Big( \int_{\Gamma_{20C_2r_i}(y')}d(x,E)^{-2m} \ud \sigma(x) \Big)^{\frac{1}{2}} \\
& \lesssim \frac{\|1_{2B_i}f\|_{L^1(\mu)}}{(20C_2r_i)^m} \lesssim M_\mu f(y) \leq \delta \lambda.
\end{split}
\end{equation*}

Combining the above estimates with \eqref{est:good lambda 1}, we have shown that there exists an absolute constant $C$ such that 
$$
\calC_{\mu,s}^t (1_{(2B_i)^c}f)(y) \leq (C\delta +1)\lambda.
$$
If $\delta(\varepsilon)$ is small enough, then this gives \eqref{step I follows}.

The cases $20C_2r_i >t$ and $20C_2r_i<s$ need only parts of the above estimates. Indeed, if $20C_2r_i>t$, then as in \eqref{est:good lambda lower} we have
$$
\calC_{\mu,s}^t(1_{(2B_i)^c}f)(y) \leq \calC_{\mu,s}^{20C_2r_i}(1_{(2B_i)^c}f)(y) \lesssim \delta \lambda,
$$
which is clearly less than $(1+\frac{\varepsilon}{2})\lambda$ for small $\delta$. If on the other hand $20C_2r_i <s$, then with the same estimates as above we get
\begin{equation*}
\begin{split}
\calC_{\mu,s}^t(1_{(2B_i)^c}f)(y) 
&\leq \big|\calC_{\mu,s}^t(1_{(2B_i)^c}f)(y)-\calC_{\mu,s}^t(1_{(2B_i)^c}f)(y')\big|+\calC_{\mu,s}^t(1_{(2B_i)^c}f)(y') \\
&\leq C \delta \lambda  +\calC_{\mu,s}^t f(y') + \calC_{\mu,s}^t(1_{2B_i}f)(y') \\
& \leq C \delta \lambda + \lambda + C\delta \lambda,
\end{split}
\end{equation*}
and this again is less than $(1+\frac{\varepsilon}{2})\lambda$ for small $\delta$. Hence \eqref{step I follows} holds in any case.

\subsection*{Step II}
Fix again some $i \in  \calI$ and consider the term
$$
\mu\big(\{y \in G_i: \calC_{\mu,s}^tf(y)>(1+\varepsilon)\lambda; M_\mu f(y) \leq \delta\lambda\}\big).
$$
We may assume that there exists a point $y_0$ in $G_i$ such that $\calC_{\mu,s}^tf(y_0)>(1+\varepsilon)\lambda$ and $M_\mu f(y_0) \leq \delta\lambda$. Step I and the weak $(1,1)$-boundedness of $\calC_{\mu,s}^t \colon \calM(E) \to L^{1,\infty}(\mu\lfloor G_i)$ give

\begin{equation*}
\begin{split}
\mu\big(\{y \in G_i\colon  &\calC_{\mu,s}^t f(y)>(1+\varepsilon)\lambda;M_\mu f(y) \leq \delta\lambda\}\big) \\
&\leq \mu\big(\{ y \in  G_i\colon  \calC_{\mu,s}^t (1_{2B_i}f)(y) >\frac{\varepsilon}{2}\lambda \}\big) \\
& \lesssim  \frac{1}{\varepsilon\lambda} \| 1_{2B_i}f\|_{L^1(\mu)} \\
&\leq \frac{\mu(5 B_i)}{\varepsilon \lambda} M_\mu f(y_0) 
\leq \mu(5B_i) \frac{\delta}{\varepsilon}.
\end{split}
\end{equation*}

\subsection*{Step III}
Finally we can finish the estimate \eqref{GL:reduction to big pieces}.  Recall that the balls $5B_i, i \in \calI$, have bounded overlap. With Step II we have  
\begin{equation*}
\begin{split}
\sum_{i \in \calI} \mu\big(\{y \in G_i: \calC_{\mu,s}^tf(y)>(1+\varepsilon)\lambda;M_\mu f(y) \leq \delta\lambda\}\big)
&\lesssim \frac{\delta}{\varepsilon} \sum_{i \in \calI} \mu(5B_i)  \\
&\lesssim  \frac{\delta}{\varepsilon} \mu(\Omega_\lambda).
\end{split}
\end{equation*}
Combining this with Equation \eqref{GL:reduction to big pieces}, we see that again if $\delta(\varepsilon)$ is small enough, then 
$$
\mu\big(\{y \in E\colon \calC_{\mu,s}^t f(y)>(1+\varepsilon)\lambda;M_\mu f(y) \leq \delta \lambda\}\big)
\leq \Big(1-\frac{\theta}{4b}\Big)\mu(\Omega_\lambda).
$$
This concludes the proof.
\end{proof}

\section{The big pieces global $Tb$ theorem}
\begin{thm}\label{the big piece Tb}
Let $B \subset E$ be a closed ball  in $E$ and assume $\mu$ is a finite Borel  measure with support in $B$. Let $\calD_B(\omega), \omega \in \Omega,$ be a family of dyadic lattices related to the ball $B$ as explained in Section \ref{preliminaries}.  Assume $b \in L^\infty(\mu)$. Suppose $c_{acc}>0$ and for any $\omega \in \Omega$ let $T_\omega$ be the union of the maximal cubes $R \in \calD_B(\omega)$ that satisfy 
$$
\Big|  \int_R b \ud \mu \Big| <c_{acc} \mu(R). 
$$
Furthermore, we assume that there exists a Borel set $H$, an exponent $s>0$ and constants $\delta_0 \in (0,1)$ and $C_0,C_1>0$ so that the following conditions hold:   
\begin{enumerate}
\item $\mu(T_\omega \cup H) \leq \delta_0 \mu(B)$ for all $\omega\in \Omega$.
\item If  $B_r$ is a closed ball of radius $r$ in $E$ and $\mu(B_r) \geq C_0 r^m$, then $B_r \subset H$.
\item $\sup_{\lambda>0} \lambda^s \mu\big(\{y \in B\setminus H: C_\mu b(y)>\lambda \}\big) \leq C_1 \mu (B)$.
\end{enumerate}

Under these assumptions there exists a set $G \subset B$ with $\mu(G)\geq \frac{1-\delta_0}{3} \mu(B)$ such that
\begin{equation}\label{bounded in big piece}
\|1_G \calC_\mu f \|_{L^2(\mu)} \lesssim \|f\|_{L^2(\mu)}
\end{equation}
holds for all $f \in L^2(\mu)$.
\end{thm}

Before the proof we recall the \emph{$b$--adapted martingales}. Suppose we are in the set-up of Theorem \ref{the big piece Tb} and let $\omega \in \Omega$. A cube $Q \in \calD_B(\omega)$ is said to be ($\omega$-) \emph{transit} if $Q \not \subset  T_\omega \cup H$. Denote the collection of transit cubes by $\calD_B^{tr}(\omega)$. If $g$ is locally $\mu$-integrable and $Q \in \calD_B(\omega)$ we denote the average of $g$ over $Q$ by 
$$
\langle g \rangle_Q:= \frac{1}{\mu(Q)} \int_Q g \ud \mu
$$
with the understanding that if $\mu(Q)=0$ then $\langle g \rangle_Q=0$.

Let $f \in L^2(\mu)$.  For the top cube $Q_B(\omega)$ define  
$$
E_{Q_B(\omega)}f := \frac{\langle f \rangle_{Q_B(\omega)}}{\langle b \rangle_{Q_B(\omega)}}b 1_{Q_B(\omega)}.
$$
For any cube  $Q \in \calD^{tr}_B(\omega)$  define
$$
\Delta_Q f:= \sum_{Q' \in ch (Q)} A_{Q'}1_{Q'},
$$ 
where 
$$
A_{Q'}
:= \left\{ \begin{array}{ll}
\Big( \frac{\langle f \rangle_{Q'}}{\langle b \rangle_{Q'}}-\frac{\langle f \rangle_Q}{\langle b \rangle_Q} \Big)b, & \textrm{if } Q' \in \calD^{tr}_B(\omega),\\
f-\frac{\langle f \rangle_Q}{\langle b \rangle_Q}b, & \textrm{if } Q' \not \in \calD^{tr}_B(\omega).\\
\end{array} \right.
$$
Here all the averages are defined with respect to the measure $\mu$. 

With these the $L^2(\mu)$-norm of $f$ can be estimated as
\begin{equation}\label{norm with martingales}
\|f\|_{L^2(\mu)} \sim \Big(  \| E_{Q_B(\omega)}f \|_{L^2(\mu)}^2+ \sum_{Q \in \calD^{tr}_B(\omega)} \|\Delta_Q f \|_{L^2(\mu)}^2\Big)^{\frac{1}{2}} ,
\end{equation}
and the function $f$ can be represented as
$$
f = E_{Q_B(\omega)}f + \sum_{Q \in \calD_B^{tr}(\omega)} \Delta_Q f,
$$
where  convergence takes place unconditionally in $L^2(\mu)$. Also, for every  $Q \in \calD^{tr}_B(\omega)$ it holds that
$$
\int_Q \Delta_Q f \ud \mu =0.
$$

For a proof of these facts see Tolsa's book \cite{ToBook}, Section 5.4.4. Notice that therein the proofs are given for the standard dyadic cubes in the plane, but the same arguments can be carried out in our setting.

Recall that the dyadic lattices $\calD_B(\omega)$ are subcollections of dyadic lattices $\calD(\omega)$ in $E$, see Section \ref{preliminaries}. Fix some $\omega_0 \in \Omega$ and write $\calD_0:= \calD(\omega_0)$. A cube $R \in \calD_0$ is called $\omega$-transit if $R \not \subset T_\omega \cup H$. Denote the collection of transit cubes in $\calD_0$ by $\calD^{tr}_0.$ So the definition of transit cubes in $\calD_0$ depends on $\omega$, but later it will be clear what the $\omega$ is.

In the $Tb$-argument we use the following lemma, whose proof is again essentially given in \cite[Lemma 5.16]{ToBook}. 

\begin{lem}\label{matrix lemma}
Suppose $\omega \in \Omega$ and $s>0$. Let $\calD_0^{tr}$ be the collection of $\omega$-transit cubes in $\calD_0$. For $R \in \calD_0^{tr}$ and $Q \in \calD^{tr}_B(\omega)$ define the numbers
$$
A_{Q,R}^s:= \frac{\ell(Q)^{\frac{s}{2}} \ell(R)^\frac{s}{2}}{D(Q,R)^{m+s}} \mu(Q)^{\frac{1}{2}} \mu(R)^{\frac{1}{2}},
$$
where $D(Q,R):= \ell(Q)+\ell(R)+d(Q,R)$. Then the matrix $\{A_{Q,R}^s\}_{Q \in \calD_B^{tr}(\omega), R \in \calD^{tr}_0}$ defines a bounded linear operator in $\ell^2$, that is, for any two sets $\{x_Q\}_{Q \in \calD^{tr}_B(\omega)}, \{y_R\}_{R \in \calD^{tr}_0}$ of non-negative real numbers we have the estimate
$$
\sum_{Q \in \calD_B^{tr}(\omega),R \in \calD^{tr}_0} A_{Q,R}^s x_Qy_R \lesssim \Big(\sum_{Q \in \calD_B^{tr}(\omega)} x_Q^2\Big)^{\frac{1}{2}} \Big( \sum_{R \in \calD^{tr}_0}y_R^2\Big)^\frac{1}{2}.
$$
\end{lem}

Now we move on to the proof of Theorem \ref{the big piece Tb}.

\begin{proof}[Proof of Theorem \ref{the big piece Tb}] Let us first give an overview of the argument. The proof is based on the idea of \emph{suppression}. First, we suppress the square function and denote it by $\tcalC_\mu$, so that $\tcalC_{\mu} b \in L^\infty(\mu)$ and $\tcalC_\mu f(y)=\calC_\mu f(y)$ for every $f \in L^2(\mu)$ and $y$ in a set $G$ with $\mu(G) \gtrsim \mu(B)$.  Moreover, this set $G$ will be in a suitable (probabilistic) way outside of the sets $T_\omega \cup H$, in the complement of which $\mu$ is of order $m$ and $b$ is accretive.  The nice properties of the set $G$ combined with $\tcalC_\mu b \in L^{\infty}(\mu)$  allow us to run a $Tb$ argument and show that
$$
\|1_G\, \tcalC_\mu f\|_{L^2(\mu)} \lesssim \|f\|_{L^2(\mu)}, \quad   f \in L^2(\mu).
$$
Since $\tcalC_\mu f=\calC_\mu f$ in $G$, we find the  set  we were after.

Next we present the details of the proof and divide the argument into a few steps.

\subsection*{Suppression}
\begin{comment}
Suppose $\calC_{\mu,1} b (y)= \infty$ for some $y \in B$. Then $\calC_{\mu,t} b(y)=\infty$ for every $t>0$. Let $y' \in B \setminus H$ and suppose $s>10|y-y'|$ and $t>s$.  Suppose $x \in \Gamma_s^t(y)$. Then we can estimate
\begin{equation*}
\begin{split}
|T_\mu b(x)| \lesssim \int_E \frac{|b(z)}{|x-z|^{m+\alpha} }\ud \mu(z)
& \sim   \int_E \frac{|b(z)|}{(|x-y|+|y-z|)^{m+\alpha}} \ud \mu(z)\\
& \sim \int_E \frac{|b(z)|}{(|x-y'|+|y'-z|)^{m+\alpha}} \ud \mu(z) \\
&\lesssim d(x,E)^{-\alpha} M^mb(y'),
\end{split}
\end{equation*}
where in the second step we used the fact that $|x-y|>10|y-y'|$. And because $y' \not \in H$, there holds that $M^m b(y') \lesssim 1$. Since also $|T_\mu b(x)| \lesssim d(x,E)^{-\alpha}$ for $x \in \Gamma(y')$, we may aplly Lemma \ref{lem:geometric problem} to get
$$
|\calC_{\mu,t}^s b(y)- \calC_{\mu,t}^sb(y') | \lesssim 1.
$$
Because $\calC_{\mu,t}^s b(y) \to \infty$, as $s \to \infty$, we see that $\calC_{\mu,t}b(y')=\infty$. So $\calC_{\mu}b$ is infinite at every point in $B \setminus H$, which is a contradiction by the weak type assumption. Thus $\calC_{\mu,1}b(y)< \infty$ for every $y \in B$, and this implies by monotone convergence that $\calC_{\mu,t}b(y) \to 0$, as $y \to \infty$.
\end{comment}

Let $\lambda_0>0$ be a big enough number to be specified later, and consider the set 
$$
S_0:=\{y \in B\colon \calC_\mu b (y)>\lambda_0\}.
$$ 
For every $y \in B$ define the numbers
$$
t(y):= \sup\{t>0: \calC_{\mu,t} b(y)>\lambda_0\}
$$
and 
$$
r(y):= \sup\{r>0: \mu(B(y,r)) \geq 11^m C_0\,r^m\},
$$
with the convention that supremum over the empty set is zero.
Since 
$$
\calC_{\mu,t}b(y) \lesssim \|b\|_{L^1(\mu)} \Big( \int_{\Gamma_t(y)} d(x,E)^{-2m} \ud \sigma(x) \Big)^\frac{1}{2}  \to 0,
$$
as $t \to \infty$, it is clear that $t(y)$ is  finite for every $y \in B$. Notice also that $r(y)$ is finite  because $\mu$ is finite.

Suppose $y \in S_0$ is such that $t(y) \geq r(y)$. By definition we have  
$$
\calC_{\mu,t(y)/2}b(y)  > \lambda_0.
$$ 
We claim that 
$$
\calC_{\mu,100t(y)}b(y) > \lambda_0/2
$$
if $\lambda_0$ is big enough, which follows from showing that
$$
\calC_{\mu, t(y)/2}^{100t(y)} b(y) \lesssim 1.
$$
If $x \in \Gamma_{t(y)/2}(y)$, then
\begin{equation}\label{annulus, transit}
\begin{split}
|T_\mu b(x)| \lesssim \int_E \frac{|b(z)|}{|x-z|^{m+\alpha}} \ud \mu (z)& \lesssim \int_E \frac{\ud \mu (z)}{(d(x,E)+|y-z|)^{m+\alpha}} \\
& \lesssim \sum_{k=0}^\infty \frac{\mu\big(B(y,2^kd(x,E))\big)}{(2^kd(x,E))^{m+\alpha}} \\
&\lesssim d(x,E)^{-\alpha}.
\end{split}
\end{equation}
In the last step it was important to observe that $d(x,E) \geq t(y)/2 \ge r(y)/2$ and 
$$
\mu(B(y,r)) \lesssim C_0 r^m
$$ 
holds for all $r \geq \frac{r(y)}{2}$.
Therefore Lemma \ref{whitney section of a cone} gives that
$$
\calC_{\mu, t(y)/2}^{100t(y)} b(y) \lesssim \sigma\big(\Gamma_{t(y)/2}^{100t(y)}(y)\big)^{\frac{1}{2}} \lesssim 1.
$$

Consider again some $y \in S_0$ such that $t(y) \geq r(y)$, and let $y' \in B_E(y, 10t(y))$. If $x \in \Gamma_{100t(y)}(y')$, then
\begin{equation*}
\begin{split}
|T_\mu  b(x)|  \lesssim \int_E \frac{\ud \mu (z)}{(d(x,E)+|y'-z|)^{m+\alpha}} 
& \sim \int_E \frac{\ud \mu (z)}{(d(x,E)+|y-z|)^{m+\alpha}} \\
& \lesssim d(x,E)^{-\alpha}
\end{split}
\end{equation*}
by Equation \eqref{annulus, transit}. Thus
\begin{equation*}
\begin{split}
|\calC_{\mu,100t(y)}b(y')-\calC_{\mu,100t(y)}b(y)|^2 
\lesssim \sigma\big(\Gamma_{100t(y)}(y')\Delta \Gamma_{100t(y)}(y)\big) \lesssim 1
\end{split}
\end{equation*}
by Lemma \ref{lem:geometric problem}. Hence we have shown that if $\lambda_0$ is large enough, then 
\begin{equation}\label{SF big near y}
B_E(y,10t(y)) \subset \{ y \in E \colon \calC_{\mu}b(y) > \lambda_0/4\}
\end{equation}
holds for all $y \in S_0$ with $t(y) \geq r(y)$.

Suppose then $y \in B$ and $r(y)>0$, and let $y' \in B_E(y,10r(y))$. Then it holds that
$$
\frac{\mu\big(B_E(y', 11r(y))\big)}{(11r(y))^m} \geq 11^{-m} \frac{\mu\big(B_E(y, r(y))\big)}{r(y)^m} \geq C_0,
$$
which shows that $B_E(y,10r(y)) \subset H$.

Now we define the suppression. Let $y \in S_0$. If $t(y) \geq r(y)$, we define
$$
A_y:=    \{x \in \Gamma(y): d(x,E) < 2t(y)\},
$$
and if $t(y)<r(y)$, we define
$$
A_y:=   \{x \in \Gamma(y): d(x,E) < r(y)\}.
$$
We also set
$$
A:=\bigcup_{y \in S_0} A_y.
$$
Since every $A_y$ is open, we see that $A$ is open and hence a Borel measurable subset of $\R^n \setminus E$. 
The suppressed kernel $\tilde{S}$ is defined as 
$$
\tilde{S}(x,y):=S(x,y)1_{\R^n \setminus A}(x).
$$
This is clearly a square function kernel satisfying the same size and $y$-H\"older conditions, and with it we define the corresponding operators $\tilde{T}_\mu$ and $\tcalC_\mu$. From the definition it follows that
\begin{equation}\label{formula for suppressed}
\tcalC_\mu f(y)=\Big(\int_{\Gamma(y) \setminus A} |T_\mu f(x)|^2 d(x,E)^{2\alpha} \ud \sigma(x) \Big)^{\frac{1}{2}}
\end{equation}
for all $f \in \bigcup_{p \in [1,\infty]}L^p(\mu)$ and $y \in E$.

Next we verify the relevant properties of the suppressed operator. To this end, define the exceptional set 
$$
S:=\bigcup_{y \in S_0} B_E\big(y,10\max (t(y),r(y))\big).
$$
Because $B_E(y,10r(y)) \subset H$ for every $y$ such that $r(y)>0$, it holds that
$$
S \setminus H= \bigcup_{\begin{substack}{y \in S_0 \\ t(y) \geq r(y)}\end{substack}} B_E(y,10t(y))  \setminus H \subset  \{ y \in E \setminus H \colon \calC_{\mu}b(y) > \lambda_0/4\}
$$
by Equation \eqref{SF big near y}. This implies by the weak type assumption (3) in Theorem \ref{the big piece Tb} that
\begin{equation*}
\mu(S \setminus H) \leq  \frac{C_1 4^s}{\lambda_0^s}  \mu(B) \leq \frac{1-\delta_0}{2}\mu(B)
\end{equation*}
if $\lambda_0$ is again big enough. We now fix a $\lambda_0$ that satisfies all the above properties, whence
\begin{equation}\label{size of exceptional}
\mu(H \cup T_\omega \cup S) \leq \mu(T_\omega \cup H)+ \mu(S \setminus H) \leq \frac{1+\delta_0}{2}\mu(B).
\end{equation}

 We claim that 
\begin{equation}\label{suppressed bounded}
\tilde{\calC}_\mu b(y) \leq \lambda_0, \quad  y \in B.
\end{equation}
Indeed, it is clear from \eqref{formula for suppressed} that 
$$
\tcalC_\mu f (y) \leq \calC_{\mu} f(y)
$$
for every $f \in \bigcup_{p \in [1,\infty]}L^p(\mu)$ and $y \in B$. Thus, by the definition of $S_0$ we have $\tcalC_\mu b(y) \leq \lambda_0$ for every $y \in B\setminus S_0$, while, if $y \in S_0$,
$$
\tcalC_\mu b(y) \leq \Big(\int_{\Gamma(y)\setminus A_y} |T_\mu b(x)|^2 d(x,E)^{2\alpha} \ud \sigma(x) \Big)^{\frac{1}{2}} \leq \lambda_0
$$
by the definitions of $A_y$ and $t(y)$.

We shall now prove that for every $f$ we have
\begin{equation}\label{outside equal}
\tilde{\calC}_\mu f(y)=\calC_\mu f(y), \quad   y \in B\setminus S.
\end{equation} 
This follows from showing that $\Gamma(y) \cap A = \emptyset$ for every $y \in B \setminus S$. To get a contradiction, suppose there exist $y \in B \setminus S$ and $x \in \Gamma(y) \cap A$. Thus, there is  $y' \in S_0$ so that $x \in A_{y'}$. Suppose first $t(y') \geq r(y')$. Then the definition of $A_{y'}$ implies that
$$
|y-y'| \leq |y-x|+|x-y'| \leq 4 d(x,E) < 8 t(y'),
$$
which is a contradiction because $y \not \in B_E(y', 10t(y'))$. Similarly we arrive at a contradiction in the case $t(y')<r(y')$. Hence we have shown that \eqref{outside equal} holds. 

We remark here that \eqref{suppressed bounded} and \eqref{outside equal}  are still valid if we replace $\tcalC_\mu$ by $\tcalC_{\mu,s}^t$ and $\calC_\mu$ by $\calC_{\mu,s}^t$, $0<s<t$.

\subsection*{The set $G$}
Fix two numbers  $k,l \in \Z, k < l,$ and write $s=\delta^l$, $t=\delta^k$. Assume $l$ is so big that $s:= \delta^l < \ell(Q_B(\omega))$ for every $\omega \in \Omega$. We consider the truncated square function $\calC_{\mu,s}^t$ and prove  the required estimate \eqref{bounded in big piece}  for this truncated operator with a bound that is independent on $s$ and $t$. This will then finish the proof. Now that the parameter $s$ is fixed, we also we fix the probability space $\Omega^l_{\log_\delta \ell(Q_{B}(\omega)) } =:\tOmega$ (see Section \ref{preliminaries}).

For every $y \in B$ define 
$$
p_0(y):= \mathbb{P}\big(\{ \omega \in \tOmega \colon y \in B \setminus ( H \cup T_\omega \cup S)\}\big).
$$
The set $G$ that we are after is defined as 
$$
G:= \{ y \in B \colon p_0(y)>\tau\}
$$
with some $\tau>0$ that will be specified soon. In other words, the set $G$ consists of those points in $B$ that have quantitatively big probability of being outside the sets $H \cup T_\omega \cup S$.

Note that since the probability space $\tOmega$ consists of only finitely many points, there is no problem with measurability when defining $p_0$ and $G$.

To estimate from below the measure of $G$, note that by Fubini and \eqref{size of exceptional},
\begin{equation*}
\begin{split}
\int_B p_0(y) \ud \mu(y) &= \int_B \int_{\tOmega}  1_{\{(\omega,y) \in \tOmega \times E \colon y \in B \setminus (H \cup T_\omega \cup S)\}} (y,\omega) \ud \mathbb{P}(\omega) \ud \mu(y) \\
&=\int_{\tOmega} \int_B   1_{\{(\omega,y) \in \tOmega \times E \colon y \in B \setminus (H \cup T_\omega \cup S)\}} (y,\omega)  \ud \mu(y) \ud \mathbb{P}(\omega) \\
&= \int_{\tOmega} \mu(B \setminus (H \cup T_\omega \cup S)) \,d\mathbb{P}(\omega) \\
&\geq \frac{1-\delta_0}{2}\mu(B).
\end{split}
\end{equation*}
On the other hand,
$$
\int_{B \setminus G} p_0(y) \ud \mu(y) \leq \tau \mu(B).
$$
Therefore, since $p_0(y) \leq 1$ for all $y \in B$,  we have that
$$
\frac{1-\delta_0}{2} \mu(B) \leq \int_B p_0(y) \ud \mu(y) \leq \tau \mu(B) + \mu(G).
$$
If we set $\tau:= \frac{1-\delta_0}{6}$ we infer that $\frac{1-\delta_0}{3} \mu(B) \leq \mu(G)$.

Notice also that if $h \colon E \to [0,\infty)$ is a Borel function,  then
\begin{equation}\label{property of G}
\begin{split}
\int_G h(y) \ud \mu(y)  &\leq \tau^{-1} \int_G h(y) p_0(y) \ud \mu (y) \\
&= \tau^{-1} \int_{\tOmega} \int _{G \setminus (H \cup T_\omega \cup S)} h(y) \ud  \mu(y) \ud \mathbb{P}(\omega) \\
&=\tau^{-1} \E_{\omega} \int _{G \setminus (H \cup T_\omega \cup S)} h(y) \ud  \mu(y),
\end{split}
\end{equation}
where in the last line we just denoted the integral over $\tOmega$ by $\E_\omega$. 

Moreover, since $G \subset B \setminus S$, it holds that 
$$
\| 1_G \,\calC_{\mu,s}^t f\|_{L^2(\mu)}^2 = \| 1_G\, \tcalC_{\mu,s}^tf\|_{L^2(\mu)}^2
$$
for all $f \in L^2(\mu)$.  For the rest of the proof  we fix a function $f \in L^2(\mu)$ and show that
$$
\| 1_G\, \tcalC_{\mu,s}^t f \|_{L^2(\mu)} \lesssim \| f \|_{L^2(\mu)}.
$$

Let $\omega \in \tOmega$ and note that $Q_B(\omega)$ must be a transit cube because $T_\omega \cup H \subsetneq B \subset Q_B(\omega)$. Hence
$$
\Big| \frac{\langle f \rangle_{Q_B(\omega)}}{\langle b \rangle_{Q_B(\omega)}}\Big| \big\| 1_G \tcalC_\mu  ( b 1_{Q_B(\omega)}) \big\|_{L^2(\mu)}
\lesssim | \langle f \rangle _{Q_B(\omega)} | \mu (B)^\frac{1}{2} \lesssim \| f \|_{L^2(\mu)},
$$
where we used the fact that $\tcalC_\mu b \in L^\infty(\mu)$. Thus, when we represent the function $f$ with martingale differences below, we may suppose that
$\int f \,d\mu = 0$.

\subsection*{A probabilistic reduction}
Here we make a certain probabilistic argument that corresponds to the reduction into good Whitney regions in \cite{MM1}.   Let  $\omega \in \tOmega$. Using  \eqref{property of G} we have
\begin{equation}\label{using G and division}
\begin{split}
 \|&1_G \tcalC^t_{\mu,s}f(y) \|_{L^2(\mu)}^2  
\\  & \leq \tau^{-1}  \E_\omega \int_{G \setminus (H \cup T_\omega \cup S)} \sum_{\begin{substack}{R \in \calD_{0} \\ s < \ell(R) \leq t}\end{substack}} 1_R(y) \int_{ \Gamma_{\ell(R)\delta}^{\ell(R)}(y)}|\widetilde{T}_\mu f(x)|^2 d(x,E)^{2\alpha} \ud \sigma(x) \ud \mu(y).
\end{split}
\end{equation}

Consider some big enough goodness parameter $r$ as in Lemma  \ref{probability of badness}, and recall the bound $\delta^{\gamma r \eta}$ for the probability of badness.  The sum over the dyadic cubes $R$ in \eqref{using G and division} may be divided into $\calD_B(\omega)$-good and -bad cubes. The corresponding term with bad cubes only satisfies
\begin{equation*}
\begin{split}
\E_\omega & \int_{G \setminus (H \cup T_\omega \cup S)} \sum_{\begin{substack}{R \in \calD_{0} \\ s < \ell(R) \leq t \\ R \text { is } \calD_B(\omega) \text{-bad}}\end{substack}} 1_R(y) \int_{ \Gamma_{\ell(R)\delta}^{\ell(R)}(y)}|\widetilde{T}_\mu f(x)|^2 d(x,E)^{2\alpha} \ud \sigma(x) \ud \mu(y) \\
& \leq \int_G \! \! \! \!\sum_{\begin{substack}{R \in \calD_{0} \\ s < \ell(R) \leq t}\end{substack}} \! \! \! \! \E_\omega 1_{\{\omega \in \tOmega \colon R \text{ is } \calD_B(\omega) \text{-bad}\}}(\omega) 1_R(y)\int_{ \Gamma_{\ell(R)\delta}^{\ell(R)}(y)} \! \! \! |\widetilde{T}_\mu f(x)|^2 d(x,E)^{2\alpha} \ud \sigma(x) \ud \mu(y) \\
& \lesssim \delta^{\gamma r \eta} \int_G \sum_{\begin{substack}{R \in \calD_{0} \\ s < \ell(R) \leq t}\end{substack}} 1_R(y) \int_{ \Gamma_{\ell(R)\delta}^{\ell(R)}(y)}|\widetilde{T}_\mu f(x)|^2 d(x,E)^{2\alpha} \ud \sigma(x) \ud \mu(y) \\
& = \delta^{\gamma r \eta} \|1_G \tcalC^t_{\mu,s}f\|_{L^2(\mu)}^2.
\end{split}
\end{equation*}
Hence, letting the goodness parameter $r$ to be big enough, we get
\begin{equation}\label{reduced to good}
\begin{split}
\|1_G &\tcalC^t_{\mu,s}f\|_{L^2(\mu)}^2  \\
&\lesssim \E_\omega \int_{G \setminus (H \cup T_\omega \cup S)} \sum_{\begin{substack}{R \in \calD_0 \\ s < \ell(R) \leq t \\ R \text{ is } \calD_B(\omega) \text{-good}}\end{substack}} 1_R(y) \int_{ \Gamma_{\ell(R)\delta}^{\ell(R)}(y)}|\widetilde{T}_\mu f(x)|^2 d(x,E)^{2\alpha} \ud \sigma(x) \ud \mu(y)\\
& \leq \E_\omega \int_{B} \sum_{\begin{substack}{R \in \calD_0 \\ s < \ell(R) \leq t \\ R \text{ is } \calD_B(\omega) \text{-good} \\ R \not \subset T_\omega \cup H}\end{substack}} 1_R(y) \int_{ \Gamma_{\ell(R)\delta}^{\ell(R)}(y)}|\widetilde{T}_\mu f(x)|^2 d(x,E)^{2\alpha} \ud \sigma(x) \ud \mu(y).
\end{split}
\end{equation}

\subsection*{A $Tb$-argument}
Fix some random parameter $\omega \in \tOmega$. Now that $\omega$ is fixed, we denote by $\calD^{tr}_0$ the collection of $\omega$-transit cubes in $\calD_0$.  By \eqref{reduced to good} it is enough prove
\begin{equation}\label{fixed omega}
\Big(\int_{B} \sum_{\begin{substack}{R \in \calD_0^{tr} \\ s<\ell(R) \leq t \\ R \text{ is } \calD_B(\omega) \text{-good} }\end{substack}} 1_R(y) \ell(R)^{2\alpha} \int_{ \Gamma_{\ell(R)\delta}^{\ell(R)}(y)}|\widetilde{T}_\mu f(x)|^2  \ud \sigma(x) \ud \mu(y) \Big)^\frac{1}{2}
\lesssim \|f\|_{L^2(\mu)},
\end{equation}
where we noticed that $d(x,E) \sim \ell(R)$ for every $x \in \Gamma_{\ell(R)\delta}^{\ell(R)}(y)$. 

Let us abbreviate 
$$
\calD':= \{R \in \calD_0^{tr} \colon s<\ell(R) \leq t , R \cap B \not= \emptyset, R \text{ is } \calD_B(\omega) \text{-good}\},
$$
whence we may replace the sum over the cubes $R$ in \eqref{fixed omega} with the sum over $R \in \calD'.$
Using  martingale differences we split the function $f$ as $f=\sum_{Q \in \calD^{tr}_B(\omega)}\Delta_Q f$. Then the estimate \eqref{fixed omega} is split into  four pieces according to the relative positions of the cubes $Q$ and $R$. For $R \in \calD'$ define the following collections: 
\begin{enumerate}
\item $\scrD_1(R):=\{ Q \in \calD^{tr}_B(\omega)\colon \ell(Q) <\delta\ell(R)\}$;
\item $\scrD_2(R):=\{Q \in \calD_B^{tr}(\omega)\colon \ell(Q) \geq \delta \ell(R) \text{ and } d(Q,R) >\ell(R)^\gamma\ell(Q)^{1-\gamma}\}$;
\item $\scrD_3(R):=\{Q \in \calD^{tr}_B(\omega)\colon \delta \ell(R) \leq \ell(Q) \leq \delta^{-r}\ell(R) \text{ and } d(Q,R) \leq \ell(R)^\gamma \ell(Q)^{1-\gamma}\}$;
\item $\scrD_4(R):= \{ Q \in \calD_B^{tr}(\omega)\colon \ell(Q) > \delta^{-r}\ell(R) \text{ and } d(Q,R) \leq \ell(R)^\gamma \ell(Q)^{1-\gamma}\}$. 
\end{enumerate} 
Next, we record with proof a few preliminary (and completely standard) estimates related to these collections. For these, recall the coefficients $A^s_{Q,R}$ from Lemma
\ref{matrix lemma}.
\begin{lem}
If $R \in \calD'$ and $Q \in \scrD_1(R)$, then for $y \in R$ and $x \in \Gamma_{\delta\ell(R)}^{\ell(R)}(y)$ we have
\begin{equation*}
\ell(R)^\alpha\big|\widetilde{T}_\mu \Delta_Q f (x)\big|
\lesssim A^\beta_{Q,R} \mu(R)^{-\frac{1}{2}} \|\Delta_Q f \|_{L^2(\mu)}.
\end{equation*}
\end{lem}

\begin{proof} 
Note that in this case it holds that
$$
d(x,Q) \geq d(x,E) \geq \delta \ell(R) > 24 \ell(Q) \geq 2\diam (Q).
$$ 
Since $\Delta_Q f$ has integral zero, we can estimate
\begin{equation*}
\begin{split}
\ell(R)^\alpha\big|\widetilde{T}_\mu \Delta_Q f (x)\big| &=\ell(R)^\alpha\Big| \int_Q \big( \tilde S(x,z)-\tilde S(x,c_Q)\big) \Delta_Q f(z) \ud \mu(z) \Big| \\
& \lesssim \ell(R)^{\alpha}\frac{\diam(Q)^\beta}{d(x,Q)^{m+\alpha+\beta}} \|\Delta_Q f\|_{L^1(\mu)} \\
&\sim \frac{\ell(R)^\alpha\ell(Q)^\beta}{\big(|x-y|+d(y,Q)\big)^{m+\alpha+\beta}} \|\Delta_Q f\|_{L^1(\mu)} \\
&\lesssim \frac{\ell(Q)^{\frac{\beta}{2}}\ell(R)^{\frac{\beta}{2}}}{\big(\ell(Q)+\ell(R)+d(Q,R)\big)^{m+\beta}} \mu(Q)^{\frac{1}{2}} \|\Delta_Q f \|_{L^2(\mu)}. 
\end{split}
\end{equation*}
\end{proof}

\begin{lem}
If $R \in \calD'$ and  $Q \in\scrD_2(R)$, then for $y \in R$ and $x \in \Gamma_{\delta\ell(R)}^{\ell(R)}(y)$ we have
$$
\ell(R)^\alpha\big|\widetilde{T}_\mu \Delta_Q f (x)\big| \lesssim A^\alpha_{Q,R} \mu(R)^{-\frac{1}{2}} \|\Delta_Q f \|_{L^2(\mu)}.
$$
\end{lem}

\begin{proof}

First do the direct estimate
\begin{equation*}
\begin{split}
\ell(R)^{\alpha} \big| \widetilde{T}_\mu \Delta_Q f (x)\big| & \lesssim \frac{\ell(R)^\alpha\|\Delta_Q f\|_{L^1(\mu)}}{d(x,Q)^{m+\alpha}} \lesssim 
\frac{\ell(R)^\alpha\|\Delta_Q f \|_{L^1(\mu)}}{\big(\ell(R)+d(Q,R)\big)^{m+\alpha}}.
\end{split}
\end{equation*}

If $\ell(Q)< d(Q,R)$, then $d(Q,R) \sim D(Q,R)$, and accordingly
$$
\frac{\ell(R)^\alpha}{\big(\ell(R) +d(Q,R)\big)^{m+\alpha}} \lesssim  \frac{\ell(Q)^{\frac{\alpha}{2}}\ell(R)^{\frac{\alpha}{2}}}{D(Q,R)^{m+\alpha}}.
$$
On the other hand if $\ell(Q)\geq d(Q,R)$, then
\begin{equation*}
\begin{split}
\frac{\ell(R)^\alpha}{\big( \ell(R) +d(Q,R)\big)^{m+\alpha}} \leq \frac{\ell(R)^\alpha}{\big(\ell(R)^\gamma \ell(Q)^{1-\gamma}\big)^{m+\alpha}} 
&= \frac{\ell(R)^{\alpha-\gamma(m+\alpha)} \ell(Q)^{\gamma(m+\alpha)}}{\ell(Q)^{m+\alpha}} \\ 
&\sim \frac{\ell(R)^{\frac{\alpha}{2}} \ell(Q)^{\frac{\alpha}{2}}}{D(Q,R)^{m+\alpha}},
\end{split}
\end{equation*}
where we took into account that $\gamma (m+\alpha)= \frac{\alpha}{2}$. Combining these estimates proves the claim.
\end{proof}

The proof of the next lemma is just a direct application of the kernel size estimate.

\begin{lem}
If $R \in \calD'$ and  $Q \in \scrD_3(R)$, then for any $y \in E$ and $x \in \Gamma_{\delta\ell(R)}^{\ell(R)}(y)$
there holds that
\begin{equation*}
\begin{split}
\ell(R)^\alpha \big| \wtT_\mu \Delta_Q f(x) \big| \lesssim \frac{\mu(Q)^\frac{1}{2} \|\Delta_Q f\|_{L^2(\mu)}}{\ell(R)^{m}}  \sim A^\alpha_{Q,R} \mu(R)^{-\frac{1}{2}} \|\Delta_Q f\|_{L^2(\mu)}.
\end{split}
\end{equation*}

\end{lem}

Using the four collections the left hand side of \eqref{fixed omega} satisfies 
$
LHS\eqref{fixed omega} \leq \sum_{i=1}^4 \Lambda_i,
$
where 
\begin{equation*}
\Lambda_i:= \Big(\int_{B} \sum_{R \in \calD'} 1_R(y) \ell(R)^{2\alpha} \int_{ \Gamma_{l(R)\delta}^{l(R)}(y)}\big|\widetilde{T}_\mu \sum_{Q \in \scrD_i(R)} \Delta_Q f(x)\big|^2  \ud \sigma(x) \ud \mu(y) \Big)^\frac{1}{2}.
\end{equation*}
Since $\sigma\big( \Gamma^{\ell(R)}_{\delta \ell(R)} (y)\big) \lesssim 1$ for every $y \in E$ and $R \in \calD'$, from the lemmas above it is seen that
\begin{equation*}
\begin{split}
\Lambda_1+\Lambda_2 + \Lambda_3 &\lesssim \Big( \sum_{R \in \calD'} \Big[ \sum_{Q \in \calD^{tr}_B(\omega)} A^\beta_{Q,R} \|\Delta_Q f \|_{L^2(\mu)} \Big]^2 \Big)^\frac{1}{2}  \\
&+\Big( \sum_{R \in \calD'} \Big[ \sum_{Q \in \calD^{tr}_B(\omega)} A^\alpha_{Q,R} \|\Delta_Q f \|_{L^2(\mu)} \Big]^2 \Big)^\frac{1}{2} 
\lesssim \|f\|_{L^2(\mu)},
\end{split}
\end{equation*}
where we applied Lemma \ref{matrix lemma} and Equation \eqref{norm with martingales} in the last step.

It only remains to estimate $\Lambda_4$.
Suppose $R \in \calD'$ and $k \in \Z$, $k \ge r$, are such that $\delta^{-k}\ell(R)  \leq \ell(Q_B(\omega))$. Since $R \cap Q_B(\omega) \not= \emptyset$ (as $R \cap B \ne \emptyset$ by the definition of $\calD'$), there exists a cube $Q \in \calD_B(\omega)$ with $\ell(Q)=\delta^{-k}\ell(R)$ so that $R \cap Q \not= \emptyset$. Then, because $R$ is $\calD_B(\omega)$-good, $k \geq r$ and $d(R,Q) = 0 \leq \ell(R)^\gamma \ell(Q)^{1-\gamma}$, it must be that $d(R,E \setminus Q) > \ell(R)^\gamma \ell(Q)^{1-\gamma}$, and in particular $R \subset Q$. We denote this unique cube $Q \in \calD_B(\omega)$ by $Q(R,k)$. Notice that
$Q(R,k)$ is transit, since $R$ is.

Note that $\scrD_4(R)$ can be non-empty only for those $R \in \calD'$ such that $\ell(R) < \delta^r \ell(Q_B(\omega))$, and
that in this case we have (using the fact that $R$ is good)
\begin{equation}\label{collection 4}
\scrD_4(R) = \Big\{Q(R,k) \colon r<k \leq \log_\delta \frac{\ell(R)}{\ell(Q_B(\omega))}\Big\}.
\end{equation}
If $R \in \calD'$ and $Q \in \scrD_4(R)$, then $Q = Q(R,k)$ for some $k > r$. For notational convenience, we denote $Q_R = Q(R, k-1)$ (which exists since $k-1 \ge r$). In other words, 
$Q_R$ is the unique child $Q' \in ch(Q)$ that still contains $R$.
\begin{lem}\label{part of IV}

Let $R \in \calD'$ and $Q \in \scrD_4(R)$. Then for $y \in R$ and $x \in \Gamma_{\delta\ell(R)}^{\ell(R)}(y)$ we have
\begin{equation*}
\ell(R)^\alpha \big| \wtT_\mu(1_{Q \setminus Q_R} \Delta_{Q} f)(x) \big| 
\lesssim A^\alpha_{Q,R} \mu(R)^{-\frac{1}{2}} \|\Delta_{Q} f \|_{L^2(\mu)}.
\end{equation*}

\end{lem}
\begin{proof}
Because $R$ is good it holds that 
\begin{equation*}
d(x,Q\setminus Q_R) \gtrsim d(y,Q\setminus Q_R) \geq \ell(R)^{\gamma}\ell(Q_R)^{1-\gamma},
\end{equation*}
and so (since $\ell(Q_R) \sim \ell(Q)+ \ell(R)$ and $d(Q,R)=0$) we have
\begin{equation*}
\begin{split}
\ell(R)^\alpha \big| \wtT_\mu(1_{Q \setminus Q_R} \Delta_Q f)(x) \big| 
&\lesssim  \frac{\ell(R)^\alpha\|\Delta_Q f\|_{L^1(\mu)}}{d(x,Q\setminus Q_R)^{m+\alpha}} \\
&\lesssim \frac{\ell(R)^{\alpha-\gamma(m+\alpha)} \ell(Q_R)^{\gamma(m+\alpha)}}{\ell(Q_R)^{m+\alpha}}\mu(Q)^\frac{1}{2} \|\Delta_Q f\|_{L^2(\mu)}\\
& \lesssim \frac{\ell(R)^{\frac{\alpha}{2}}\ell(Q)^{\frac{\alpha}{2}}}{D(Q,R)^{m+\alpha}}\mu(Q)^{\frac{1}{2}} \|\Delta_Q f \|_{L^2(\mu)}.
\end{split}
\end{equation*}
\end{proof}

Using Lemma \ref{part of IV} we see that
\begin{equation*}
\begin{split}
\Big(\int_{B} \sum_{R \in \calD'} &1_R(y) \ell(R)^{2\alpha} \int_{ \Gamma_{l(R)\delta}^{l(R)}(y)}\big|\widetilde{T}_\mu \sum_{Q \in \scrD_4(R)} 1_{Q \setminus Q_R}\Delta_Q f(x)\big|^2  \ud \sigma(x) \ud \mu(y) \Big)^\frac{1}{2} \\
& \lesssim \Big( \sum_{R \in \calD'} \Big[ \sum_{Q \in  \calD^{tr}_B(\omega)} A^\alpha_{Q,R} \|\Delta_Q f \|_{L^2(\mu)} \Big]^2 \Big)^\frac{1}{2} 	
\lesssim \|f\|_{L^2(\mu)},
\end{split}
\end{equation*}
where Lemma \ref{matrix lemma} was applied in the last step.
Therefore, it remains to bound
\begin{equation}\label{main term}
\Big(\int_{B} \sum_{R \in \calD'} 1_R(y) \ell(R)^{2\alpha} \int_{ \Gamma_{l(R)\delta}^{l(R)}(y)}\big|\widetilde{T}_\mu \sum_{Q \in \scrD_4(R)} 1_{Q_R}\Delta_Q f(x)\big|^2  \ud \sigma(x) \ud \mu(y) \Big)^\frac{1}{2}.
\end{equation}

Fix for the moment some $R \in \calD'$ and $Q \in \scrD_4(R)$. Recall that the cube $Q_R$ is transit. Hence we can write
$$
1_{Q_R}\Delta_Q f =B_{Q,R}b- B_{Q,R}b1_{E \setminus Q_R},
$$
where
$$
B_{Q,R}:=\frac{\langle f \rangle_{Q_R}}{\langle b \rangle _{Q_R}} - \frac{\langle f \rangle_Q}{\langle b \rangle_Q}.
$$
Let $y \in R$ and $x \in \Gamma_{\delta\ell(R)}^{\ell(R)}(y)$, and look at the term $\big| B_{Q,R} \wtT_\mu(b1_{E \setminus Q_R})(x)\big|$. For every $z \in E \setminus Q_R$ we have 
$$
|x-z| \gtrsim |y-z| \gtrsim  \ell(R)^\gamma\ell(Q_R)^{1-\gamma} +|y-z|, 
$$
where the fact that $R$ is good was used. Hence
\begin{equation*}
\begin{split}
\big|\wtT_\mu(b1_{E \setminus Q_R} )(x) \big| \lesssim \int_{E \setminus Q_R} \frac{|b(z)|}{|x-z|^{m+\alpha}} \ud \mu(z) & \lesssim \int_E \frac{\ud \mu(z)}{(\ell(R)^\gamma\ell(Q_R)^{1-\gamma}+|y-z|)^{m+\alpha}} \\
&\lesssim (\ell(R)^\gamma\ell(Q_R)^{1-\gamma})^{-\alpha},
\end{split}
\end{equation*}
where the last estimate follows from the usual calculations as in \eqref{annulus, transit}, but it is important to use the facts that $R$ is transit, $y \in R$ and $\ell(R)^\gamma\ell(Q_R)^{1-\gamma} \geq \ell(R)$. Because $Q_R$ is  transit  there holds that
$$
|B_{Q,R}| \lesssim \Big|B_{Q,R} \frac{1}{\mu(Q_R)} \int_{Q_R} b \ud \mu\Big| = \big| \langle \Delta_Qf \rangle_{Q_R}\big| \leq \frac{1}{\mu(Q_R)^\frac{1}{2}}. \|1_{Q_R} \Delta_Q f \|_{L^2(\mu)}.
$$
A combination of these estimates gives
\begin{equation}\label{error term}
\ell(R)^\alpha \big| \wtT_\mu(1_{E \setminus Q_R} B_{Q,R} b) (x)\big| \lesssim \frac{\ell(R)^{\alpha(1-\gamma)}}{\ell(Q_R)^{\alpha(1-\gamma)}}\frac{\|1_{Q_R} \Delta_Q f \|_{L^2(\mu)}}{\mu(Q_R)^\frac{1}{2}}.
\end{equation}

Applying \eqref{error term} we have
\begin{equation}\label{E setminus Q_R}
\begin{split}
\int_{B} &\sum_{R \in \calD'} 1_R(y) \ell(R)^{2\alpha} \int_{ \Gamma_{l(R)\delta}^{l(R)}(y)}\big|\sum_{Q \in \scrD_4(R)}  B_{Q,R} \wtT_\mu(1_{E \setminus Q_R}b)(x)\big|^2  \ud \sigma(x) \ud \mu(y) \\
&\lesssim   \sum_{R \in \calD'} \mu(R) \Big[ \sum_{Q \in \scrD_4(R)} \frac{\ell(R)^{\alpha(1-\gamma)}}{\ell(Q_R)^{\alpha(1-\gamma)}}\frac{\|1_{Q_R} \Delta_Q f \|_{L^2(\mu)}}{\mu(Q_R)^\frac{1}{2}}\Big]^2.
\end{split}
\end{equation}
Since 
$$
\sum_{Q \in \scrD_4(R)} \frac{\ell(R)^{\alpha(1-\gamma)}}{\ell(Q_R)^{\alpha(1-\gamma)}}\lesssim 1
$$
by \eqref{collection 4}, we can continue with Jensen's inequality as
\begin{equation*}
\begin{split}
RHS\eqref{E setminus Q_R}& \lesssim    \sum_{R \in \calD'} \mu(R)  
\sum_{Q \in \scrD_4(R)} \frac{\ell(R)^{\alpha(1-\gamma)}}{\ell(Q_R)^{\alpha(1-\gamma)}}\frac{\|1_{Q_R} \Delta_Q f \|_{L^2(\mu)}^2}{\mu(Q_R)} \\
&= \sum_{k=r}^\infty \delta^{k\alpha(1-\gamma)} \sum_{\begin{substack}{Q \in \calD^{tr}_B(\omega)\\ Q \not= Q_B(\omega)}\end{substack}}\frac{\|1_Q \Delta_{\widehat{Q}} f \|_{L^2(\mu)}^2}{\mu(Q)} \sum_{\begin{substack}{R \in \calD' \\ Q(R,k)=  Q }\end{substack}} \mu(R) \\
& \leq \sum_{k=r}^\infty \delta^{k\alpha(1-\gamma)} \sum_{\begin{substack}{Q \in \calD^{tr}_B(\omega)\\ Q \not= Q_B(\omega)}\end{substack}}\|1_Q \Delta_{\widehat{Q}} f \|_{L^2(\mu)}^2 \lesssim \| f \|_{L^2(\mu)}^2.
\end{split}
\end{equation*}
Combining this with \eqref{main term} we see that the last term to be estimated is 
\begin{equation}\label{para}
I:=\int_{B} \sum_{R \in \calD'} 1_R(y) \ell(R)^{2\alpha} \int_{ \Gamma_{l(R)\delta}^{l(R)}(y)}\big|\sum_{Q \in \scrD_4(R)}  B_{Q,R} \wtT_\mu b(x)\big|^2  \ud \sigma(x) \ud \mu(y). \end{equation}
So far we have used only the properties of the kernel $S$, but in this part we finally apply the fact that $\tcalC_\mu b \in L^\infty(\mu)$. 

Suppose $R \in \calD'$ is such that $\ell(R) < \delta^r \ell(Q_B(\omega))$ (i.e. the case $\scrD_4(R) \ne \emptyset$). Then  
\begin{equation}\label{telescope}
\begin{split}
\Big|\sum_{Q \in \scrD_4(R)} B_{Q,R}\Big|= \Big|\sum_{\begin{substack}{ Q \in \calD_B^{tr}(\omega) \\ Q \supsetneq Q(R,r)}\end{substack}} \Big(\frac{\langle f \rangle_{Q_R}}{\langle b \rangle _{Q_R}} - \frac{\langle f \rangle_Q}{\langle b \rangle_Q}\Big) \Big|&= \Big|\frac{\langle f \rangle_{Q(R,r)}}{\langle b \rangle_{Q(R,r)}}-\frac{\langle f \rangle_{Q_B(\omega)}}{\langle b \rangle_{Q_B(\omega)}}\Big| \\
& \lesssim |\langle f \rangle_{Q(R,r)}|,
\end{split}
\end{equation}
because the cubes $Q(R,k)$, $k \ge r$, are transit and $\langle f \rangle_{Q_B(\omega)} = 0$.

Using \eqref{telescope} we have
\begin{equation*}\begin{split}
I&=  \int_{B} \sum_{\begin{substack}{R \in \calD'\\ \ell(R)<\delta^r\ell(Q_B(\omega))}\end{substack}}1_R(y)\ell(R)^{2\alpha} \int_{ \Gamma_{l(R)\delta}^{l(R)}(y)}\Big|\frac{\langle f \rangle_{Q(R,r)}}{\langle b \rangle_{Q(R,r)}}\wtT_\mu b(x)\Big|^2  \ud \sigma(x) \ud \mu(y) \\
&\lesssim \int _B \sum_{Q \in \calD^{tr}_B(\omega)} |\langle f \rangle_{Q}|^2  \sum_{\begin{substack}{R \in \calD' \\ Q(R,r)=Q}\end{substack}} 1_{R}(y) \int_{ \Gamma_{l(R)\delta}^{l(R)}(y)}\big| \wtT_\mu b(x)\big|^2 d(x,E)^{2\alpha} \ud \sigma(x) \ud \mu(y) \\
&\leq  \sum_{Q \in \calD_B^{tr}(\omega)} \big|\langle f \rangle_Q\big|^2 a_Q,
\end{split}
\end{equation*}
where 
$$
a_Q:= \int _B \sum_{\begin{substack}{R \in \calD' \\ Q(R,r)=Q}\end{substack}} 1_{R}(y) \int_{ \Gamma_{l(R)\delta}^{l(R)}(y)}\big| \wtT_\mu b(x)\big|^2d(x,E)^{2\alpha}  \ud \sigma(x) \ud \mu(y).
$$

The last remaining thing is to verify the dyadic Carleson condition for the numbers $a_Q$.  Luckily, this is straightforward because of our strong testing condition $\wtT_\mu b \in L^\infty(\mu)$. Indeed, fix some $Q_0 \in \calD_B(\omega)$. Then
\begin{equation*}
\begin{split}
\sum_{\begin{substack}{Q \in \calD_B^{tr}(\omega) \\  Q \subset Q_0}\end{substack}}a_Q & 
\le \int_B \sum_{\begin{substack}{R \in \calD' \\  R \subset Q_0} \end{substack}} 1_R(y) \int_{ \Gamma_{l(R)\delta}^{l(R)}(y)}\big| \wtT_\mu b(x)\big|^2  d(x,E)^{2\alpha} \ud \sigma(x) \ud \mu(y) \\
& \leq \int_{Q_0} \tcalC_\mu b (y)^2 \ud \mu(y) \lesssim \mu(Q_0).
\end{split}
\end{equation*}
Because the Carleson condition holds, we have
$$
\sum_{Q \in \calD_B^{tr}(\omega)} \big|\langle f \rangle_Q\big|^2 a_Q \lesssim \|f\|_{L^2(\mu)}^2.
$$
This was the last piece in the $Tb$ argument and concludes the proof of Theorem \ref{the big piece Tb}.
\end{proof}

\appendix
\section{A sketch of the proof of Lemma \ref{probability of badness}}\label{appendix}

Here we sketch the proof of Lemma \ref{probability of badness} following the arguments in \cite{AH}, Theorem 2.11. The constants $M$ and $L$ in the statement of Lemma \ref{probability of badness, appendix} are related to properties of $E$ as a geometrically doubling space, and are used in the construction of random dyadic cubes.

\begin{lem}\label{probability of badness, appendix}
There exist two constants $C=C(M,L)>0$ and $ \eta \in (0,1]$ so that the following holds. Fix some big enough (depending on $\gamma$) goodness parameter $r$. Suppose $B \subset E$ is a ball in $E$ and construct the dyadic lattices $\calD(\omega), \omega \in \Omega,$ in $E$ using the center of $B$ as the fixed reference dyadic point, see Section \ref{preliminaries}. For some fixed $\omega_0 \in \Omega$ write $\calD_0=\calD(\omega_0)$ and recall the lattices $\calD_B(\omega) \subset \calD(\omega)$.   

Assume $k_0 \in \Z$ is such that $ \ell(Q_B(\omega))=\delta^{k_0}$ for some, and hence for every, $\omega \in \Omega$.  Let $k_1 \in \Z$ be any number such that $k_1 \geq k_0+r$. With the fixed $\omega_0$ define the probability space
$$
\Omega_{k_0}^{k_1} := \{ \omega \in \Omega \colon \omega(m)= \omega_0(m) \text{ if } m<k_0 \text{ or } m>k_1\}
$$
equipped with the natural probability measure  such that the coordinates $m \mapsto \omega(m), k_0 \leq m \leq k_1,$ are independent and uniformly distributed over $\{0, \dots, L\} \times \{1, \dots, M\}$.

Then, for every cube $R \in \calD_0$ with $\ell(R) \geq \delta^{k_1} $ it holds that
\begin{equation}
\mathbb{P}\big(\{ \omega \in \Omega^{k_1}_{k_0} : R \text{ is } \calD_B(\omega) \text{-bad}\}\big) \leq C \delta^{\gamma r \eta}.
\end{equation}
\end{lem}

Before the proof define for $\varepsilon>0$ the $\varepsilon$-\emph{boundary} $\partial_\varepsilon Q$ of a cube $Q \in \calD(\omega)$ by
$$
\partial_\varepsilon Q:=\{ y \in Q \colon d(y, E \setminus Q) <\varepsilon \ell(Q)\}. 
$$

\begin{proof}[Proof of Lemma \ref{probability of badness, appendix}]
Let $R \in \calD_{0}$ be such that $\ell(R)= \delta^m$, where $k_0 +r \leq m \leq k_1$ (if $m < k_0 + r$, then $R$ is automatically good by definition).
 Fix some $\omega \in \Omega_{k_0}^{k_1}$ for the moment. First we show that if $R$ is $\calD_B(\omega)$-bad, then there exists  $l \in \Z, k_0 \leq l \leq m-r$, and $Q \in \calD_l(\omega)$ so that  $c_R \in \partial _{7\ell(R)^\gamma/\ell(Q)^\gamma }Q$. Indeed, let $k_0 \leq l \leq m-r$ and suppose $Q \in \calD_l(\omega)$ is such that $c_R \in Q$. If 
 $ c_R \not \in \partial _{7\ell(R)^\gamma/\ell(Q)^\gamma }Q$, 
 then because $R \subset B(c_R, 6\ell(R))$ by \eqref{new centers}, it is seen that
 $$
d(R,E\setminus Q) \geq \ell(R)^\gamma\ell(Q)^{1-\gamma}.
$$
In particular, we have
$$
\max\big(d(R,Q'), d(R,E\setminus Q')\big) \ge \ell(R)^\gamma\ell(Q')^{1-\gamma}
$$
for all $Q' \in \calD_B(\omega)$ with $\ell(Q')=\delta^l$. If this happens for all $l$ such that $k_0 \leq l \leq  m-r$, then the cube $R$ is $\calD_B(\omega)$-good. We have shown that 
\begin{equation}\label{prob. of R bad}
\mathbb{P}\big(\{ \omega \in \Omega^{k_1}_{k_0} : R \text{ is } \calD_B(\omega) \text{-bad}\}\big)
\leq \sum_{l=k_0}^{m-r} \mathbb{P}\big(\{ \omega \in \Omega^{k_1}_{k_0} : c_R \in  \bigcup_{Q \in \calD_l(\omega)} \partial _{7\ell(R)^\gamma/\ell(Q)^\gamma }Q\}\big).
\end{equation}

Next we fix some $l \in \{k_0, \dots, m-r\}$ and estimate the corresponding term in the right hand side of \eqref{prob. of R bad}. Following the argument in \cite{AH}, if $\omega(p)$ has a certain value for some $p \geq l$ such that $\delta^p \geq 49 \delta^{m\gamma} \delta^{l(1-\gamma)}$, then 
\begin{equation}\label{x not near boundary}
c_R \not \in  \bigcup_{ Q \in \calD_l(\omega)} \partial_{7\ell(R)^\gamma/\ell(Q)^\gamma} Q.
\end{equation}
This part of the argument is short, but to state it we would need to introduce more of the construction of the random dyadic cubes. We refer the reader to \cite{AH}.

The requirement $\delta^l \geq \delta^p \geq 49 \delta^{m\gamma} \delta^{l(1-\gamma)}$  amounts to 
\begin{equation}\label{p between}
l \leq p \leq \log_\delta 49 +m\gamma +l(1-\gamma).
\end{equation}
Note that in particular every such $p$ satisfies $k_0 \leq p \leq k_1$. For \eqref{p between} to make sense we demand $r$ to be so big that $r\gamma>2$, say,  whence 
\begin{equation}\label{how many p}
\log_\delta 49 +m\gamma +l(1-\gamma)-l\geq (m-l)\gamma-1>1.
\end{equation} 
Denote by $\lfloor \log_\delta 49 +(m-l)\gamma \rfloor$ the smallest integer less than or equal to $\log_\delta 49 +(m-l)\gamma$.

For every $p \in \Z$ the variable $\omega(p)$ has the probability $\tau:= \frac{1}{M(L+1)}$ of getting a given value. Hence, by \eqref{x not near boundary}, we have
 $$
 \mathbb{P}\big(\{ \omega \in \Omega^{k_1}_{k_0} : c_R \in  \bigcup_{Q \in \calD_l(\omega)} \partial_{7\delta^{(m-l)\gamma}} Q\}\big)
\leq (1-\tau)^{\lfloor \log_\delta 49 +(m-l)\gamma \rfloor} \leq C (1-\tau)^{(m-l)\gamma}.
$$
Combining this with \eqref{prob. of R bad} we get
$$
\mathbb{P}\big(\{ \omega \in \Omega^{k_1}_{k_0} : R \text{ is } \calD_B(\omega) \text{-bad}\}\big)\leq C\sum_{l=k_0}^{m-r} (1-\tau)^{ (m-l)\gamma} \sim (1-\tau)^{r\gamma},
$$
and we can rewrite the bound as
$$
(1-\tau)^{r\gamma}= \delta^{r\gamma\log_\delta (1-\tau)} .
$$  
This gives the required conclusion because $\log_\delta (1-\tau) \in (0,1)$.
\end{proof}

\section{$L^2$ boundedness implies weak $(1,1)$ boundedness}\label{weak (1,1)-boundedness}

We verify here that if $\mu$ is a measure of order $m$ in $E$ and  $\calC_\mu$ is bounded in $L^2(\mu)$, then
$$
\calC \colon \calM(E) \to L^{1, \infty}(\mu)
$$
boundedly. The proof of this follows the standard steps using the Calder\'on-Zygmund decomposition, but we check the details because of our unusual set-up. 

First we record a few lemmas, and begin with the non-homogeneous Calder\'on-Zygmund decomposition whose proof can be found for example in \cite{ToBook}. We say that a collection $\{B_i\}_i$ of balls in $\R^n$ has \emph{bounded overlap} if there exists a constant $C$ such that
$$
\sum_i 1_{B_i}(x) \leq C
$$
for every $x \in \R^n$.

\begin{lem}\label{CZ-dec.}
Let $\mu$ be a Radon measure in $\R^n$ and suppose $\nu$ is a complex measure in $\R^n$ with compact support. Let $\lambda > 2^{n+1}\frac{|\nu|(\R^n)}{\mu(\R^n)}$. 

There exists a countable family $\{B_i\}_{i \in \calI}$ of closed balls with bounded overlap and with centers in $\supp \nu$, and a function $f \in L^1(\mu)$ with $\| f \|_{L^\infty(\mu)} \leq \lambda$ so that
\begin{gather}
|\nu|(B_i) > 2^{-n-1} \lambda \mu(2B_i) \ \text{ for all }i, \label{CZ1}\\
|\nu|(\eta B_i) \leq 2^{-n-1} \lambda \mu(2\eta B_i) \ \text{ for all }i \text{ and }\eta>2, \label{CZ2}\\
1_{\R^n \setminus \bigcup_i B_i}\nu= f \mu. \label{CZ3}
\end{gather}

For every $ i \in \calI$ suppose $R_i$ is $(6,6^{m+1})$-doubling ball  (with respect to $\mu$)  concentric with $B_i$ and $r(R_i) > 4r(B_i)$. Define the functions $w_i:= \frac{1_{B_i}}{\sum_{j} 1_{B_j}}$. Then there exists a family $\{\varphi_i\}_{i \in \calI}$ of functions such that each function is of the form $\varphi_i=\alpha_i h_i$, where $\alpha_i \in \C$ and $h_i$ is a non-negative function, and this family satisfies the properties
\begin{gather}
\supp \varphi_i \subset R_i,  \label{CZ4}\\
\int \varphi_i \ud \mu= \int w_i \ud \nu, \label{CZ5}\\
\sum_{i \in \calI} |\varphi_i| \leq C_1 \lambda, \label{CZ6} \\
\| \varphi_i\|_{L^{\infty}(\mu)} \mu(R_i) \leq c |\nu|(B_i). \label{CZ7}
\end{gather}
Here $C_1$ is a constant depending on $m$ and $n$, and $c$ is an absolute constant. 
\end{lem}

The next two simple lemmas can also be found for example in \cite{ToBook}.

\begin{lem}\label{big doubling balls}
Suppose $\mu$ is a measure of order $m$ in $\R^n$. Let $b > a^m$. If $B$ is a ball in $\R^n$, then there exists $s>1$ such that the ball $sB$ is $(a, b)$-doubling.
\end{lem}

\begin{lem}\label{no doubling balls between}
Suppose $\mu$ is a Radon measure in $\R^n$. Let $b > a^m, a>1$. Suppose $B_1$ and $B_2$ are two balls in $\R^n$ with center $x$  and  $B_1 \subset B_2$. Assume none of the balls $a^kB_1$ is $(a, b)$-doubling for those $k \in \Z$ such that $ B_1 \subsetneq a^k B_1 \subset B_2$. Then
$$
\int_{B_2 \setminus B_1} \frac{1}{|y-x|^m} \ud \mu(y) \lesssim_{a, b,m} \frac{\mu(B_2)}{r(B_2)^m}.
$$
\end{lem}

\begin{thm}\label{weak (1,1)}
Let $\mu$ be a measure of order $m$ in $E$.  Suppose that $\calC_\mu$ is bounded in $L^2(\mu)$. Then
$$
\calC \colon \calM(E) \to L^{1,\infty}(\mu)
$$
is bounded with a constant depending on the kernel parameters, the dimension $n$ and the $L^2(\mu)$ norm of $\calC_\mu$.
\end{thm}

\begin{proof}
Let $\nu \in \calM(E)$ and $\lambda>0$. We want to show that
$$
\mu\big(\{ y \in E \colon \calC\nu(y) > \lambda \}\big) \lesssim \frac{1}{\lambda} |\nu|(E).
$$
We may assume that $\lambda>2^{n+1} \frac{|\nu|(E)}{\mu(E)}$, since otherwise we have nothing to prove.  

Suppose first that $\nu$ is compactly supported. We apply Lemma \ref{CZ-dec.}  with $\lambda$ to the measures $\nu$ and $\mu$ to get a function $f$ and an almost disjoint collection $\{B_i\}$ of closed balls with centers in $E$ such that \eqref{CZ1}, \eqref{CZ2} and \eqref{CZ3} hold. For each $i$, let $R_i \supsetneq 4B_i$ be the smallest $(8,8^{m+1})$-doubling closed ball concentric with $B_i$, which exists by Lemma \ref{big doubling balls}. We apply Lemma \ref{CZ-dec.} with the balls $R_i$ to get a collection $\{\varphi_i\}$ of functions such that \eqref{CZ4}, \eqref{CZ5}, \eqref{CZ6} and \eqref{CZ7} hold.

Using the balls $B_i$ and the functions $\varphi_i$ we can write the measure $\nu$ as 
\begin{equation*}
\begin{split}
\nu=1_{\R^n \setminus \bigcup_iB_i}\nu+ 1_{\bigcup_i B_i} \nu &= f \mu +\sum_i w_i \nu \\
&= \big(f+\sum_i \varphi_i)\mu + \sum_i (w_i\nu - \varphi_i\mu).
\end{split}
\end{equation*}
Write $g= f+\sum_i\varphi_i$ and $b= \sum_i b_i=\sum_i( w_i\nu-\varphi_i\mu)$. Then we have
\begin{equation}\label{good and bad}
\mu\big(\{y \in E \colon \calC \nu(y)> \lambda\}\big) \leq \mu\big(\{y \in E \colon \calC_\mu g > \frac{\lambda}{2}\}\big)+ \mu\big(\{ y \in E \colon \calC b > \frac{\lambda}{2}\}\big).
\end{equation}

The $L^2(\mu)$-boundedness of $\calC_\mu$ and the fact that $\|f +\sum_i   \varphi_i\|_{L^{\infty}(\mu)} \lesssim \lambda$  by \eqref{CZ6} give
$$
\mu\big(\{y \in E \colon \calC_\mu g > \frac{\lambda}{2}\}\big) \lesssim \frac{1}{\lambda^2}  \| f +\sum_i \varphi_i \|_{L^2(\mu)}^2  
\lesssim \frac{1}{\lambda}  \| f +\sum_i \varphi_i \|_{L^1(\mu)}.
$$
Using Equations \eqref{CZ3}, \eqref{CZ4}, \eqref{CZ7}  and the bounded overlapping property of $\{B_i\}$ we get
\begin{equation*}
\begin{split}
\| f +\sum_i \varphi_i \|_{L^1(\mu)}  \leq \int |f| \ud \mu+ \sum_i \int |\varphi_i| \ud \mu & \leq |\nu|(\R^n) +\sum_i \|\varphi_i\|_{L^\infty(\mu)} \mu(R_i) \\
&\lesssim |\nu|(\R^n)+ \sum_i |\nu|(B_i) \\ 
&\lesssim |\nu|(\R^n).
\end{split}
\end{equation*}

Next, we consider the  second term  on the right hand side of \eqref{good and bad}. Note that
$$
\mu\big( \bigcup_i 2B_i\big) \leq \sum_i \mu(2B_i) \leq \frac{2^{n+1}}{\lambda}\sum_i |\nu|(B_i) \lesssim \frac{1}{\lambda} |\nu|(\R^n).
$$
Hence  we need to show that
$$
\mu\big(\{y \in E \setminus \bigcup_i 2B_i \colon \calC b > \frac{\lambda}{2}\}\big) \lesssim \frac{1}{\lambda} |\nu|(\R^n).
$$
First estimate as
\begin{equation*}
\begin{split}
\mu\big(\{y \in E \setminus \bigcup_i 2B_i \colon \calC b > \frac{\lambda}{2}\}\big) 
& \leq \frac{2}{\lambda} \int_{E \setminus \bigcup_i 2B_i} \calC b \ud \mu \\
& \leq \frac{2}{\lambda} \sum_i \int_{E \setminus 2B_i} \calC b_i \ud \mu.
\end{split}
\end{equation*}
We will prove that
\begin{equation}\label{single bad term}
\int_{E \setminus 2B_i} \calC b_i \ud \mu \lesssim |\nu| (B_i)
\end{equation}
holds for every $i$, which then concludes the proof because $\sum_i |\nu| (B_i) \lesssim |\nu|( \R^n)$.

Fix some $i$, and recall the ball $R_i$ related to the ball $B_i$. We begin the proof of \eqref{single bad term} by writing
$$
\int_{E \setminus 2B_i} \calC b_i \ud \mu  = \int_{E \setminus 8R_i} \calC b_i \ud \mu+ \int_{8R_i \setminus 2B_i} \calC b_i \ud \mu=: I + II.
$$
We consider the term $I$ first. Let $y \in E \setminus 8R_i$ and $x \in \Gamma(y)$. Let $c_{R_i}$ be the center of $R_i$, whence it follows that $|x-c_{R_i}| \geq 2r(R_i)$. Since $\supp b_i \subset R_i$ and $b_i(R_i)=0$, we may apply the $y$-continuity \eqref{eq:yHol} of the square function kernel to get
\begin{equation*}
\begin{split}
|Tb_i(x)| \lesssim \frac{r(R_i)^\beta }{|x-c_{R_i}|^{m+\alpha+\beta}} |b_i|(R_i)  \sim \frac{r(R_i)^\beta}{(|x-y|+ |y-c_{R_i}|)^{m+\alpha+\beta} } |b_i| (R_i).
\end{split}
\end{equation*}
Also
\begin{equation*}
\begin{split}
|b_i | (R_i) = |w_i \nu - \varphi_i\mu| (R_i) &\leq  \int_{R_i} w_i \ud |\nu| + \int_{R_i} |\varphi_i| \ud \mu  \\
 &\leq |\nu|(B_i) +\| \varphi_i \|_{L^\infty(\mu)} \mu(R_i) \lesssim |\nu|(B_i).
\end{split}
\end{equation*}
Thus
\begin{equation*}
\begin{split}
\calC b_i(y)^2 &\lesssim \int_{\Gamma(y)} \frac{r(R_i)^{2\beta} d(x,E)^{2\alpha}}{(|x-y|+|y-c_{R_i}|)^{2(m+\alpha+\beta)}} \ud \sigma(x) |b_i |(R_i)^2 \\
& \lesssim \frac{r(R_i)^{2\beta}}{|y-c_{R_i}|^{2(m+\beta)}} |\nu|(B_i)^2,
\end{split}
\end{equation*}
where we applied Lemma \ref{cone computation} in the second step. Because $\mu$ is of order $m$ we get
\begin{equation*}
\begin{split}
I =\int_{E \setminus 8R_i} \calC b_i \ud \mu 
&\lesssim  \int_{E \setminus 8R_i} \frac{r(B_i)^{\beta}}{|y-c_{R_i}|^{m+\beta}}  \ud \mu(y) |\nu|(B_i) \\
& \lesssim |\nu|(B_i)|.
\end{split}
\end{equation*}

It remains to consider the term $II$. Since $b_i= w_i \nu - \varphi_i \mu$ we have
\begin{equation*}
II=\int_{8R_i \setminus 2B_i} \calC b_i \ud \mu \leq \int_{8R_i \setminus 2B_i} \calC (w_i\nu) \ud \mu+\int_{8R_i \setminus 2B_i} \calC_\mu\varphi_i \ud \mu =: II_1+II_2.
\end{equation*}

The $L^2(\mu)$-boundedness of $\calC_\mu$ gives
\begin{equation*}
\begin{split}
II_2  \leq \mu(8R_i) ^\frac{1}{2} \| \calC_\mu\varphi_i \|_{L^2(\mu)}& \lesssim \mu(R_i)^\frac{1}{2} \| \varphi_i \|_{L^2(\mu)} \leq \| \varphi_i \|_{L^\infty(\mu)} \mu(R_i) \lesssim |\nu| (B_i),
\end{split}
\end{equation*}
where we used  the fact that $R_i$ is $(8,8^{m+1})$-doubling and the properties \eqref{CZ4} and \eqref{CZ7} of the function $\varphi_i$. 
 
Finally, we consider the term $II_1$. Suppose $y \in 8R_i \setminus 2B_i$ and $x \in \Gamma(y)$. Then, by Lemma \ref{distance to a point in E},
$$
|T (w_i\nu)(x)| \lesssim  \int_{B_i} \frac{\ud |\nu|(z)}{|x-z|^{m+\alpha}} \sim \frac{|\nu|(B_i)}{(|x-y|+|y-c_{B_i}|)^{m+\alpha}},
$$
and hence
\begin{equation*}
\begin{split}
\calC(w_i\nu)(y) &\lesssim \Big( \int_{\Gamma(y)} \frac{d(x,E)^{2\alpha}}{(|x-y|+|y-c_{B_i}|)^{2(m+\alpha)}} \ud \sigma (x) \Big)^{\frac{1}{2}} |\nu|(B_i) \\
&\lesssim \frac{1}{|y-c_{B_i}|^{m} }|\nu|(B_i).
\end{split}
\end{equation*} 
Integrating this over $8R_i \setminus 2B_i$ gives
\begin{equation*}
\begin{split}
II_1=\int_{8R_i \setminus 2B_i} \calC(w_i\nu) \ud \mu & \lesssim \int_{8R_i \setminus R_i} \frac{1}{|y-c_{B_i}|^{m}} \ud \mu(y)|\nu|(B_i) \\
&+  \int_{R_i \setminus 2B_i} \frac{1}{|y-c_{B_i}|^{m}} \ud \mu(y)|\nu|(B_i)  \\ 
&\lesssim |\nu|(B_i),
\end{split}
\end{equation*}
where we used Lemma \ref{no doubling balls between} to estimate the integral over $R_i \setminus 2B_i$. This finishes the proof \eqref{single bad term}, and hence also the proof Theorem \ref{weak (1,1)} in the case when $\nu$ is compactly supported.

Suppose then $\nu$ is not compactly supported but $\mu$ is compactly supported. Suppose $M>0$ is such that $\supp \mu \subset B(0,M/2)$. Write $\tilde{\nu}:= \nu \lfloor ( E \setminus B(0,M))$. Then for any $y \in \supp \mu$ and $x \in \Gamma(y)$ we have
\begin{equation*}
\begin{split}
\big|T\tilde{\nu}(x)\big| &\lesssim \int_{E \setminus B(0,M)} \frac{\ud |\nu|(z)}{(|x-y|+|y-z|)^{m+\alpha}} \\ 
&\leq  \frac{|\nu| (\R^n)}{\big(|x-y|+d(y,E \setminus B(0,M))\big)^{m+\alpha}},
\end{split}
\end{equation*}
and this gives by Lemma \ref{cone computation} that
\begin{equation*}
\begin{split}
\calC \tilde{\nu}(y) &\lesssim  \Big(\int_{\Gamma(y)} \frac{d(x,E)^{2\alpha}}{\big(|x-y|+d(y,E \setminus B(0,M))\big)^{2(m+\alpha)}}\ud \sigma(x) \Big)^\frac{1}{2}|\nu| (\R^n) \\
& \lesssim \frac{1}{d(y,E \setminus B(0,M))^{m}} |\nu|(\R^n).
\end{split}
\end{equation*}
Hence, if $M$ is big enough, we get
\begin{equation*}
\begin{split}
\mu\big(\{y \in E \colon \calC \nu (y) > \lambda\}\big) &\leq \mu\big(\{y \in E \colon \calC \big(\nu\lfloor  B(0,M)\big)(y) > \frac{\lambda}{2}\} \big)\\
& \lesssim \frac{1}{\lambda} |\nu| \big(B(0,M)\big) \leq \frac{1}{\lambda} |\nu|(\R^n),
\end{split}
\end{equation*}
where the second inequality holds because $\nu \lfloor B(0,M)$ is compactly supported.

Suppose finally that neither $\nu$ nor $\mu$ is compactly supported. Then for  every $M>0$ it holds that
\begin{equation*}
\begin{split}
\mu\lfloor B(0,M) \big(\{ y \in E \colon \calC \nu (y) > \lambda \}\big)
\lesssim \frac{1}{\lambda} |\nu|(\R^n),
\end{split}
\end{equation*}
because $\mu \lfloor B(0,M)$ is compactly supported. Letting $M$ tend to infinity concludes the proof.

\end{proof}

\end{document}